\documentclass[preprint]{imsart}
\RequirePackage[OT1]{fontenc}
\RequirePackage{amsthm,amsmath,amsfonts}
\RequirePackage[numbers]{natbib}
\RequirePackage[colorlinks,citecolor=blue,urlcolor=blue]{hyperref}
\usepackage{graphicx}

\startlocaldefs

\theoremstyle{plain}
\newtheorem{property}{Property}
\newtheorem{theorem}{Theorem}
\newtheorem{lemma}{Lemma}

\newtheorem{cor}{Corollary}
\newtheorem{rem}{Remark}
\newtheorem{ex}{Example}
\theoremstyle{definition}
\newtheorem{definition}{Definition}

\DeclareMathOperator{\tr}{tr}
\DeclareMathOperator{\df}{df}
\DeclareMathOperator{\vect}{vec}

\DeclareMathOperator{\diag}{Diag}
\DeclareMathOperator*{\argmin}{argmin}

\DeclareMathOperator*{\converge}{\longrightarrow}

\newcommand{\R}{\mathbb{R}}

\newcommand{\NN}{\mathbb{N}}
\newcommand{\PP}{\mathbb{P}}
\newcommand{\Proba}{\mathbb{P}}
\newcommand{\LL}{L^2(\PP)}

\newcommand{\X}{\mathcal{X}}
\newcommand{\C}{\mathcal{C}}
\newcommand{\F}{\mathcal{F}}

\newcommand{\D}{\mathcal{D}}
\newcommand{\N}{\mathcal{N}}
\newcommand{\M}{\mathcal{M}}
\renewcommand{\P}{\mathcal{P}}

\newcommand{\e}{\varepsilon}
\newcommand{\s}{\sigma}

\newcommand{\g}{\gamma}
\renewcommand{\a}{\alpha}

\renewcommand{\th}{\theta}
\renewcommand{\b}{\beta}

\renewcommand{\d}{\delta}

\renewcommand{\S}{\Sigma}
\renewcommand{\l}{\lambda}

\renewcommand{\hat}[1]{\widehat{#1}}
\renewcommand{\tilde}[1]{\widetilde{#1}}

\newcommand{\esp}[1]{\mathbb{E} \left[ #1 \right]}
\newcommand{\empesp}[1]{\hat{\mathbb{E}} \left[ #1 \right]}
\newcommand{\linspan}[1]{\textrm{span} \left\{ #1 \right\}}

\newcommand{\norr}[1]{\left\| #1 \right\|_2 }

\newcommand{\paren}[1]{\left( #1 \right)} 
 
\newcommand{\croch}[1]{\left[ #1 \right]} 
 
\newcommand{\set}[1]{\left\{ #1 \right\}}

\newcommand{\absj}[1]{\left\lvert #1 \right\rvert} 

\providecommand{\norm}[1]{\left \lVert #1 \right\rVert}

\newcommand{\scal}[2]{\langle #1 , #2 \rangle}

\newcommand{\egaldef}{:=}

\newcommand{\trsp}{^{\top}}

\newcommand{\pt}{\enspace .}
\newcommand{\virg}{\enspace ,}
\newcommand{\ptvirg}{\enspace ;}
\newcommand{\et}{~\textrm{ and }~}

\newcommand{\fh}{\widehat{f}}
\newcommand{\Fh}{\widehat{F}}

\newcommand{\Mh}{\widehat{M}}

\newcommand{\MT}{\mathrm{MT}}
\newcommand{\ST}{\mathrm{ST}}

\newcommand{\SD}{\mathrm{SD}}
\newcommand{\AV}{\mathrm{AV}}

\newcommand{\MSD}{M_{\SD}}
\newcommand{\MAV}{M_{\AV}}
\newcommand{\HK}{$\textrm{H}_{\textrm{K}}(\b)$}
\newcommand{\HypM}{$\textrm{H}_{\textrm{M}}(\b,\d)$}
\newcommand{\HAV}{$\textrm{H}_{\textrm{AV}}(\d,C_1,C_2)$}
\newcommand{\kbd}{\kappa(\b,\d)}

\newcommand{\matheq}[1]{$\displaystyle{#1}$}

\newcommand{\setn}{\{1,\dots,n\}}
\newcommand{\setp}{\{1,\dots,p\}}

\endlocaldefs

\begin{document}
  \begin{frontmatter}
  \title{Comparison bewteen multi-task and single-task oracle risks in kernel ridge regression}
  \runtitle{Comparison between multi-task and single-task oracles}
  \begin{aug}
    \author{\fnms{Matthieu} \snm{Solnon} \ead[label=e1]{matthieu.solnon@ens.fr}}
    \address{ ENS; Sierra Project-team\\
      D\'epartement d'Informatique de l'\'Ecole Normale Sup\'erieure \\
      (CNRS/ENS/INRIA UMR 8548)\\
      23, avenue d'Italie, CS 81321 \\
      75214 Paris Cedex 13, France \\
      \printead{e1}}
    \runauthor{M. Solnon}
    \affiliation{DI-ENS}
  \end{aug}
  \begin{abstract}
    In this paper we study multi-task kernel ridge regression and try to understand when the multi-task procedure performs better than the single-task one, in terms of averaged quadratic risk. 
    In order to do so, we compare  the risks of the estimators with perfect calibration, the \emph{oracle risk}.
    We are able to give explicit settings, favorable to the multi-task procedure, where the multi-task oracle performs better than the single-task one.
    In situations where the multi-task procedure is conjectured to perform badly, we also show the oracle does so. 
    We then complete our study with simulated examples, where we can compare both oracle risks in more natural situations. 
    A consequence of our result is that the multi-task ridge \emph{estimator} has a lower risk than any single-task estimator, in favorable situations. 	
  \end{abstract}
\tableofcontents
\end{frontmatter}

\section{Introduction}

Increasing the sample size is the most common way to improve the performance of  statistical estimators. 
In some cases (see, for instance, the experiments of \citet{Evgeniou} on customer data analysis or those of  \citet{jacob:clustered} on molecule binding problems), having access to some new data may be impossible, often due to experimental limitations.
One way to circumvent those constraints is to use datasets from several related (and, hopefully, ``similar'') problems, as if it gave additional (in some sense) observations on the initial problem. 
The statistical methods using this heuristic are called ``multi-task'' techniques, as opposed to ``single-task'' techniques, where every problem is treated one at a time. 
In this paper, we study kernel  ridge regression in a multi-task framework and try to understand when multi-task can improve over single-task.

The first trace of a multi-task estimator can be found in the work of \citet{stein1956inadmissibility}. 
In this article, Charles Stein showed that the usual maximum-likelihood estimator of the mean of a Gaussian vector (of dimension larger than 3, every dimension representing here a task) is not admissible---that is, there exists another estimator that has a lower risk for every parameter. 
He showed the existence of an estimator that uniformly attains a  lower quadratic risk by shrinking the estimators along the different dimensions towards an arbitrary point. 
An explicit form of such an estimator was given by \citet{james1961estimation}, yielding the famous James-Stein estimator.  
This phenomenon, now known as the ``Stein's paradox'', was widely studied in the following years and the behaviour of this estimator was confirmed by empirical studies, in particular the one from \citet{efron1977stein}. 
This first example clearly shows the goals of the multi-task procedure: an advantage is gained by borrowing information from different tasks (here, by shrinking the estimators along the different dimensions towards a common point), the improvement being scored by the global (averaged) squared risk. 
Therefore, this procedure does not guarantee individual gains on every task, but a global improvement on the sum of those task-wise risks. 

\medskip

We consider here $p \geq 2$ different regression tasks, a framework we refer to as ``multi-task'' regression, and where the performance of the estimators is measured by the fixed-design quadratic risk.
Kernel ridge regression is a classical framework to work with and comes with a natural norm, which often has desirable properties (such as, for instance, links with regularity).
This norm is also a natural ``similarity measure'' between the regression functions. 
\citet{Evgeniou} showed how to extend kernel ridge regression to a multi-task setting, by adding a regularization term that binds the regression functions along the different tasks together. 
One of the main questions that is asked is to assert whether the multi-task estimator has a lower risk than any single-task estimator. 
It was recently proved by \citet{solnon:hal-00610534} that a fully data-driven calibration of this procedure is possible, given some assumptions on the set of matrices used to regularize---which correspond to prior knowledge on the tasks. 
Under those assumptions, the estimator is showed to verify an \emph{oracle inequality}, that is, its risk matches (up to constants) the best possible one, the \emph{oracle risk}.
Thus, it suffices to compare the oracle risks for the multi-task procedure and the single-task one to provide an answer to this question.

\medskip

The multi-task regression setting, which could also be called ``multivariate regression'',  has already been studied in different papers. 
It was first introduced by  \citet{Brown_Zidek_1980} in the case of ridge regression, and then adapted by \citet{Evgeniou} in its kernel form. 
Another view of the meaning of ``task similarity'' is that the functions all share a few common features, and can be expressed by a similar regularization term. 
This idea was expressed in a linear set up (also known as group lasso) by \citet{Obozinski_Wainwright_Jordan_2011} and \citet{Lounici:1277345}, in multiple kernel learning by \citet{koltchinskii2010sparsity} or in semi-supervised learning by \citet{Ando:2005:FLP:1046920.1194905}.
The kernel version of this was also studied \citep{DBLP:journals/ml/ArgyriouEP08,jacob:clustered}, a convex relaxation leading to a trace norm regularization and allowing the calibration of parameters. 
Another point of view was brought by \citet{ben2003exploiting}, defining a multi-task framework in classification, two classification problems being similar if, given a group of permutations of the input set, a dataset of the one can be permuted in a dataset of the other. 
They followed the analysis of \citet{Baxter2000}, which shows very general bounds on the risk of a multi-task estimator in a model-selection framework, the sets of all models reflecting the insight the statistician has on the multi-task setting.   

\medskip

Advantages of the multi-task procedure over the single task one were first shown experimentally in various situations  by, for instance, \citet{Thrun96g}, \citet{Caruana:1997:ML:262868.262872} or \citet{Bakker:2003:TCG:945365.945370}. 
For classification, \citet{ben2003exploiting} compare upper bounds on multi-task and single-task classification errors, and showed that the multi-task estimator could, in some settings, need less training data to reach the same upper bounds. 
The low dimensional linear regression setting was analysed by \citet{2009arXiv0912.5338R}, who showed that, under sparsity assumptions, restricted isometry conditions and using the trace-norm regularization, the multi-task estimator achieves the rates of a single-task estimator with a $np$-sample.
\citet{liang10regularization} also obtained a theoretical  criterion, applicable to the linear regression setting and unfortunately non observable, which tells when the multi-task estimator asymptotically has a lower risk than the lower one. 
A step was recently carried by \citet{FeldmanGF12MTav} in a kernel setting where every function is estimated by a constant. 
They give a closed-form expression of the oracle for two tasks and run simulations to compare the risk of the multi-task estimator to the risk of the single-task estimator. 

\medskip

In this article we study the oracle multi-task risk and compare it to the oracle single-task risk. 
We then find situations where the multi-task oracle is proved to have a lower risk than the single-task oracle. 
This allows us to better understand which situation favors the multi-task procedure and which does not. 
After having defined our model (Section~\ref{sec.model}), we write down the risk of a general multi-task ridge estimator and see that it admits a convenient decomposition using two key elements: the mean of the tasks and the resulting variance (Section~\ref{sec.risk}). 
This decomposition allows us to optimize this risk and get a precise estimation of the oracle risk, in settings where the ridge estimator is known to be minimax optimal (Section~\ref{sec.MTrisk}). 
We then explore several repartitions of the tasks that give the latter multi-task rates, study their single-task oracle risk (Section~\ref{sec.STrisk}) and compare it to their respective multi-task rates. 
This allows us to discriminate several situations, depending whether the multi-task oracle either outperforms its single-task counterpart, underperforms it or whether both behave similarly (Section~\ref{sec.MTSTcomp}). 
We also show that, in the cases favorable to the multi-task oracle detailed in the previous sections, the estimator proposed by \citet{solnon:hal-00610534} behaves accordingly and achieves a lower risk than the single-task oracle (Section~\ref{sec.riskest}). 
We finally study settings where we can no longer explicitly study the oracle risk, by running simulations, and we show that the multi-task oracle continues to retain the same virtues and disadvantages as before (Section~\ref{sec.simus}).    

\medskip

\noindent We now introduce some notations, which will be used throughout the article.

\begin{itemize}
\item The integer $n$ is the sample size, the integer $p$ is the number of tasks.
\item For any $n \times p$ matrix $Y$, we define
\begin{equation*}
   y = \vect(Y) \egaldef \paren{ Y_{1,1}, \ldots, Y_{n,1}, Y_{1,2}, \ldots, Y_{n,2}, \ldots, Y_{1,p}, \ldots, Y_{n,p} } \in \R^{np}, 
\end{equation*}
that is, the vector in which the columns $Y^j \egaldef (Y_{i,j})_{1 \leq i \leq n}$ are stacked.
\item $\mathcal{M}_n(\R)$ is the set of all real square matrices of size $n$.
\item $\mathcal{S}_p(\R)$ is the set of symmetric matrices of size $p$.
\item $\mathcal{S}_p^+(\R)$ is the set of  symmetric positive-semidefinite matrices of size $p$.
\item $\mathcal{S}_p^{++}(\R)$ is the set of symmetric positive-definite matrices of size $p$.
\item $\mathbf{1}$ is the vector of size $p$ whose components are all equal to $1$.
\item $\norr{\cdot}$ is the usual Euclidean norm on $\R^{k}$ for any $k \in \NN$: $\forall u \in \R^k$, $\norr{u}^2 \egaldef \sum_{i=1}^k u_i^2$.
\item For two real sequences $(u_n)$ and $(v_n)$ we write $u_n \asymp v_n$ if there exists positive constants $\ell$ and $L$ such that, for a large enough $n$, $\ell v_n \leq u_n \leq L v_n$.
\item For $(a,b)\in\R^2$, $a\wedge b$ denotes the minimum of $a$ and $b$.
\end{itemize}
\section{Kernel ridge regression in a multi-task setting \label{sec.model}}

  We consider here that each task is treated as a kernel ridge-regression problem and we will then extend the single-task ridge-regression estimator in a multi-task setting. 

\subsection{Model and estimator}
Let $\Omega$ be a set, $\mathcal{A}$ be a $\s$-algebra on $\Omega$ and $\PP$ be a probability measure on $\mathcal{A}$. 
We observe $\D_n = (X_i,Y_i^1,\ldots,Y_i^p)_{i=1}^n \in (\X \times  \R^p)^{n}$. 
For each task $j \in \{ 1,\ldots,p\}$, $\D_n^j = (X_i,y_i^j)_{i=1}^n$ is a sample with distribution $\P^j$, whose first marginal distribution is $\P$, for which a simple regression problem has to be solved. 

We assume that for every $j\in\setp,~F^j \in \LL$, $\S$ is a symmetric positive-definite matrix of size $p$ such that the vectors $(\e_i^j)_{j=1}^p$ are independent and identically distributed (i.i.d.) with normal distribution $\N(0,\S)$, with mean zero and covariance matrix~$\Sigma$, and
\begin{equation} \label{modele}
  \forall i \in \{ 1,\ldots,n\},\forall j \in \{ 1,\ldots,p\},  ~  y_i^j = F^j(X_i) + \e_i^j \pt
\end{equation} 
We suppose here, for simplicity, that $\S = \s^2 I_p$, with $\s^2 \in \R_+^{\star}$. 

\begin{rem}
This implies that the outputs of every task are independent, which slightly simplifies the setting but allow lighter calculations. 
It is to be noted, though, that the analysis carried afterwards can still take place without this assumption. 
This can be dealt by diagonalizing $\S$, majoring the quantities of interest using the largest eigenvalue of $\S$ and minoring those quantities by its smallest eigenvalue. 
The comparisons shown in Section~\ref{sec.MTSTcomp} are still valid, only being enlarged by the condition number of $\S$. A fully data-driven estimation of~$\S$ was proposed by  \citet{solnon:hal-00610534}.
\end{rem}

We consider here a fixed-design setting, that is, we consider the input points as fixed and want to predict the output of the functions $F^j$ on those input points only. 
The analysis could be transfered to the random-design setting by using tools developped by \citet{hsu2011analysis}.

For an estimator $(\Fh^1,\dots,\Fh^p)$, the natural quadratic risk to consider is 
\begin{equation*}
  \esp{\frac{1}{np}\sum_{j=1}^p\sum_{i=1}^n(\Fh^j(X_i)-F^j(X_i))^2 | (X_1,\dots,X_n)} \pt
\end{equation*}
For the sake of simplicity, all the expectations that follow will implicitly be written conditional on $(X_1,\dots,X_n)$. 
This corresponds to the fixed-design setting, which treats the input points as fixed.   
\begin{rem}
 We will use the following notations from now on :
\begin{equation*}
 f = \vect\paren{(f^j(X_i))_{i,j}},~ f^j = \vect\paren{(f^j(X_i))_{i=1}^n} ~\textrm{ and }~ y = \vect\paren{(Y_i^j)_{i,j}} \virg
\end{equation*}
so that, when using such vectorized notations, the elements are stacked task by task, the elements refering to the first task always being stored in the first entries of the vector, and so on.
\end{rem}

We want to estimate $f$ using elements of a particular function set. 
Let $\F\subset\LL$ be a reproducing kernel Hilbert space (RKHS) \citep{Aronszajn}, with kernel~$k$ and feature map $\Phi: \X \to \F$, which give us the positive semidefinite kernel matrix $K = (k(X_i,X_{\ell}))_{1 \leq i , \ell \leq n} \in \mathcal{S}_n^+(\R)$.

As done by \citet{solnon:hal-00610534} we extend the multi-task estimators generalizing the ridge-regression used in \citet{Evgeniou}. Given a positive-definite matrix $M \in \mathcal{S}_p^{++}(\R)$, we consider the estimator 

\begin{equation}\label{min.multi.M}
  \Fh_M \in \argmin_{g \in \F^p} \Bigg\{ \underbrace{\frac{1}{np} \sum_{i=1}^n \sum_{j=1}^p (y_i^j - g^j(X_i))^2}_{\textrm{Empirical risk}} + \underbrace{\sum_{j=1}^p \sum_{\ell=1}^p M_{j,l}\scal{g^j}{g^{\ell}}_{\F}}_{\textrm{Regularization term}} \Bigg\} \pt
\end{equation}
This leads to the fixed-design estimator
\begin{equation*}
  \fh_M = A_{M}y \in \R^{np}\virg
\end{equation*}
with
\begin{equation*}
A_{M}  = A_{M,K} 
\egaldef \tilde{K}_{M}(\tilde{K}_{M}+np I_{np})^{-1} = (M^{-1} \otimes K)\left((M^{-1} \otimes K)+np I_{np}\right)^{-1} \virg
\end{equation*}
where $\otimes$ denotes the Kronecker product (see the textbook of \citet{Horn_Johnson_Matrix_analysis} for simple properties of the Kronecker product).
\begin{rem}
 This setting also captures the single-task setting. 
Taking $j\in\setp$, $f^j = (f^j(X_1),\dots,f^j(X_n))\trsp$ being the target-signal for the $j$th task and $y^j = (y_1^j,\dots,y_n^j)\trsp$ being the observed output of the $j$th task, the single-task estimator for the $j$th task becomes (for $\l\in\R_+$)
 \begin{equation*}
  \fh^j_{\l} = A_{\l}y^j = K(K+n\l I_n)^{-1} y^j\pt
 \end{equation*}

\end{rem}

\subsection{Two regularization terms for one problem}
A common hypothesis that motivates the use of multi-task estimators is that all the target functions of the different tasks lie in a single cluster (that is, the $p$ functions that are estimated are all close with respect to the norm defined on~$\F$). 
Two different regularization terms are usually considered in this setting:

\begin{itemize}
 \item one that penalizes the norms of the $p$ function and their differences, introduced by \citet{Evgeniou}, leading to the criterion (with $(g^j)_{j=1}^p \in \F^p,~ (\a,\b) \in (\R_+)^2$)
  \begin{equation}\label{eq.MSD}
   \frac{1}{np} \sum_{i=1}^n \sum_{j=1}^p (y_i^j - g^j(X_i))^2 + \frac{\a}{p} \sum_{j=1}^p \norm{g^j}_{\F}^2 + \frac{\b}{2p} \sum_{j=1}^p \sum_{k=1}^p \norm{g^j-g^k}_{\F}^2 \ptvirg
  \end{equation}
 \item one that penalizes the norms of the average of the $p$ functions and the resulting variance, leading to the criterion (with $(g^j)_{j=1}^p \in \F^p,~ (\l,\mu) \in (\R_+)^2$)
  \begin{equation} \label{eq.MAV}
   \frac{1}{np} \sum_{i=1}^n \sum_{j=1}^p (y_i^j - g^j(X_i))^2 + \l \norm{\frac{\sum_{j=1}^p g^j}{p}}_{\F}^2 + \mu\croch{\frac{\sum_{j=1}^p \norm{g^j}_{\F}^2}{p}-\norm{\frac{\sum_{j=1}^p g^j}{p}}_{\F}^2} \pt
  \end{equation}
\end{itemize}
As we will see, those two penalties are closely related. Lemma~\ref{lemma.M} indeed shows that the two former penalties can be obtained as a special case of Equation~\eqref{min.multi.M}, the matrix $M$ being respectively 
\begin{equation*} 
\MSD(\a,\b) \egaldef  \frac{\a}{p} \frac{\boldsymbol{1}\boldsymbol{1}\trsp}{p} + \frac{\a+p\b}{p}\paren{I_p-\frac{\boldsymbol{1}\boldsymbol{1}\trsp}{p}}
\end{equation*}
and
\begin{equation*}
 \MAV(\l,\mu) \egaldef \frac{\l}{p}\frac{\boldsymbol{1}\boldsymbol{1}\trsp}{p} + \frac{\mu}{p}\paren{I_p-\frac{\boldsymbol{1}\boldsymbol{1}\trsp}{p}}\pt
\end{equation*}
Thus, we see that those two criteria are related, since $\MSD(\a,\b) = \MAV(\a,\a+p\b)$ for every $(\a,\b)$. Minimizing Equations~\eqref{eq.MSD} and \eqref{eq.MAV} over $\F^p$ respectly give the ridge estimators $\fh_{SD}(\a,\b) = A_{\MSD(\a,\b)}Y$ and $\fh_{\AV}(\l,\mu) = A_{\MAV(\l,\mu)} Y$. 
\begin{rem}
 We can now see that the regularization terms used in Equations~\eqref{eq.MSD} and \eqref{eq.MAV} are equivalent when the parameters are not constrained to be positive. However, if one desires to use the regularization \eqref{eq.MSD} (that is, with $\l = \a$ and $\mu = \a+p\b$) and seeks to calibrate those parameters  by taking them to be nonnegative (which is to be expected if they are seen as regularization parameters), the following problems could occur:
 \begin{itemize}
  \item if the optimization is carried over $(\l,\mu)$, then the selected parameter $\b = \frac{\mu-\l}{p}$ may be negative;
  \item conversely, if the risk of the estimator defined by Equation~\eqref{eq.MSD} is optimized over the parameters $(\a,\a+p\b)$ with the constraints $\a \geq 0$ and $\b \geq 0$, then the infimum over $\R_+^2$ could never be approached.
 \end{itemize}
\end{rem}
 We will also show in the next section that the risk of $ \fh_{\AV}(\l,\mu)$ nicely decomposes in two parts, the first part depending only on $\l$ and the second only on $\mu$, which is not the case for $\fh_{SD}(\a,\b)$ because of the aforementioned phenomenon. 
This makes us prefer the second formulation and use the matrices $\MAV$ instead of the matrices $\MSD$.

\section{Decomposition of the risk \label{sec.risk}}

A fully data-driven selection of the hyper-parameters was proposed by \citet{Arl_Bac:2009:minikernel_long}, for the single-task ridge estimator, and by \citet{solnon:hal-00610534} for the multi-task estimator. 
The single-task estimator is shown to have a risk which is close to the single-task oracle-risk (with a fixed-design)
\begin{equation*}
 \mathfrak{R}^{\star}_{\ST} = \inf_{(\l^1,\dots,\l^p)\in\R_+^p}\set{\frac{1}{np}\esp{\sum_{j=1}^p \norr{\fh^j_{\l^j}-f^j}^2}} \virg	
\end{equation*}
while the multi-task estimator is shown to have a risk which is close to the multi-task oracle risk 
\begin{equation*}
 \mathfrak{R}^{\star}_{\MT} =  \inf_{(\l,\mu)\in\R_+^2}\set{\frac{1}{np}\esp{\norr{\fh_{\MAV(\l,\mu)} -f}^2}} \pt
\end{equation*}
The purpose of this paper is to closely study both oracle risks and, ultimately, to compare them.  
We show in this section how to decompose the risk of an estimator obtained by minimizing Equation~\eqref{eq.MAV} over $(g^j)_{j=1}^p \in \F^p$. 
A key point of this analysis is that the matrix  $\MAV(\l,\mu)$ naturally decomposes over two orthogonal vector-subspaces of $\R^p$. 
By exploiting this decomposition we can simply use the classical bias-variance decomposition to analyse the Euclidean risk of those linear estimators. 

\subsection{Eigendecomposition of the matrix $\MAV(\l,\mu)$}

In this section we show that all the matrices $\MAV(\l,\mu)$ have the same eigenvectors, which gives us a simple decomposition of the matrices $A_{\MAV(\l,\mu)}$. 
Let us denote by $(e_1,\dots,e_p)$ the canonical basis of $\R^p$. The eigenspaces of $p^{-1}\boldsymbol{1}\boldsymbol{1}\trsp$ are orthogonal and correspond to:
\begin{itemize}
  \item $\linspan{e_1+\dots+e_p}$ associated to eigenvalue $1$,
  \item $\linspan{e_2-e_1,\dots,e_p-e_1}$  associated to eigenvalue $0$.
\end{itemize}
Thus, with $(\l,\mu)\in(R^+)^2$, we can diagonalize in an orthonormal basis any matrix $\MAV(\l,\mu)$ as $M = \MAV(\l,\mu)=P\trsp D_{\frac{\l}{p},\frac{\mu}{p}}P$, with $D  = \diag\{\frac{\l}{p},\frac{\mu}{p},\dots,\frac{\mu}{p}\} = D_{\frac{\l}{p},\frac{\mu}{p}}$. 
Let us also diagonalize $K$ in an orthonormal basis : $K = Q\trsp \Delta Q$, $\Delta = \diag\{\gamma_1,\dots,\gamma_n\}$. Then
\begin{equation*}
A_M = A_{\MAV(\l,\mu)} = (P\trsp \otimes Q\trsp) \left[(D^{-1}\otimes \Delta)\left((D^{-1}\otimes \Delta)+npI_{np} \right)^{-1} \right](P \otimes Q) \pt
\end{equation*}
We can then note that $(D^{-1}\otimes \Delta)\left((D^{-1}\otimes \Delta)+npI_{np} \right)^{-1}$ is a diagonal matrix, whose  diagonal entry of index $(j-1)n+i$ ($i\in\{1,\dots,n\}$, $j\in\{1,\dots,p\}$) is 
\begin{equation*}
\left\{ \begin{array}{lr} &\frac{\g_i}{\g_i+n\l} \textrm{ if } j=1\virg\\
                          &\frac{\g_i}{\g_i+n\mu} \textrm{ if } j>1\pt \end{array}\right.
\end{equation*}
In the following section we will use the following notations : 
\begin{itemize}
 \item for every $j\in\setp$, $(h^j_i)_{i=1}^n$ denotes the coordinates of $(f^j(X_i))_{i=1}^n$ in the basis that diagonalizes $K$,
 \item for every $i\in\setn$, $(\nu^j_i)_{j=1}^p$ denotes the coordinates of $(h^j_i)_{j=1}^p$ in the basis that diagonalizes $M$.
\end{itemize}
Or, to sum up, we have : 
\begin{equation*}
\forall j\in\{1,\dots,p\},~\begin{pmatrix} h^j_1\\\vdots\\h^j_n \end{pmatrix} = Q \begin{pmatrix} f^j(X_1)\\\vdots\\f^j(X_n) \end{pmatrix}
\end{equation*}
and 
\begin{equation*}
\forall i\in\setn,~\begin{pmatrix} \nu^1_i\\\vdots\\\nu^p_i \end{pmatrix} = P ~\begin{pmatrix} h^1_i\\\vdots\\h^p_i \end{pmatrix} \pt
\end{equation*}
With the usual notation $\nu^j = (\nu^j_1,\dots,\nu^j_n)\trsp$ and $f$ , we get, by using elementary properties of the Kronecker product,
\begin{equation*}
 \nu = \begin{pmatrix} \nu^1\\\vdots\\\nu^p \end{pmatrix} = (P\otimes Q) f \pt
\end{equation*}

\subsection{Bias-variance decomposition}
We now use a classical bias-variance decomposition of the risk of $\fh_{\AV}(\l,\mu)$ and show that the quantities introduced above allow a simple expression of this risk. 
For any matrix $M \in \mathcal{S}_p^{++}(\R)$, the classical bias-variance decomposition for the linear estimator $\fh_M = A_My$ is 
\begin{align*}
  \frac{1}{np}\esp{\norr{\fh_{M} -f}^2} &= \frac{1}{np}\|(A_{M}-I_{np})f\|_2^2 + \frac{1}{np}\tr(A_{M}\trsp A_{M}\cdot(\S \otimes I_n))  \\
					&= \underbrace{\frac{1}{np}\|(A_{M}-I_{np})f\|_2^2}_{\textrm{Bias}} +\underbrace{\frac{\s^2}{np} \tr(A_{M}\trsp A_{M})}_{\textrm{Variance}} \pt  
\end{align*}
We can now compute both bias and variance of the estimator $\fh_{\AV}(\l,\mu)$ by decomposing $A_{\MAV(\l,\mu)}$ on the eigenbasis introduced in the previous section. 
\begin{description}
\item[$np\times$Variance :]

\begin{align*}
&\s^2\tr(A_{M}\trsp A_{M})  \\
=~~&\s^2 \tr\bigg( (P \otimes Q)\trsp\left[(D^{-1}\otimes \Delta)\left((D^{-1}\otimes \Delta)+npI_{np} \right)^{-1}\right]^2  (P \otimes Q)  \bigg) \\
                                          =~~&\s^2 \tr\bigg(\left[(D^{-1}\otimes \Delta)\left((D^{-1}\otimes \Delta)+npI_{np} \right)^{-1}\right]^2\bigg)\\
                                          =~~&\s^2 \sum_{i=1}^n \left[\left(\frac{\g_i}{\g_i+n\l}\right)^2+(p-1)\left(\frac{\g_i}{\g_i+n\mu}\right)^2\right] \pt
\end{align*}
\item[$np\times$Bias :]
\begin{align*}
  &\|(A_{M}-I_{np})f\|_2^2 \\
 =~~&  \|(P \otimes Q)\trsp \left[(D^{-1}\otimes K)\left((D^{-1}\otimes K)+npI_{np} \right)^{-1} -I_{np}\right](P \otimes Q)f\|_2^2 \\
                         =~~& \|\left[(D^{-1}\otimes \Delta)\left((D^{-1}\otimes \Delta)+npI_{np} \right)^{-1} -I_{np}\right] \nu\|_2^2 \\
                         =~~& (n\l)^2\sum_{i=1}^n\frac{(\nu_i^1)^2}{(\g_i+n\l)^2} +(n\mu)^2 \sum_{i=1}^n\sum_{j=2}^p \frac{(\nu^j_i)^2}{(\g_i+n\mu)^2} \\
                         =~~& (n\l)^2\sum_{i=1}^n\frac{(\nu_i^1)^2}{(\g_i+n\l)^2} +(n\mu)^2 \sum_{i=1}^n \frac{\sum_{j=2}^p(\nu^j_i)^2}{(\g_i+n\mu)^2} \pt
\end{align*}
\end{description}

\noindent Thus, the risk of  $\fh_{\AV}(\l,\mu)$ becomes
\begin{equation}\label{optim.ind}
  \begin{split}
    n\l^2\sum_{i=1}^n\frac{\frac{(\nu_i^1)^2}{p}}{(\g_i+n\l)^2} + \frac{\s^2}{np} \sum_{i=1}^n \left(\frac{\g_i}{\g_i+n\l}\right)^2    ~~~~~~~~~~~~~~~~~~~~~~~~~~~~~~~~~~~~~~~~~~\\+  n\mu^2 \sum_{i=1}^n \frac{\frac{\sum_{j=2}^p(\nu^j_i)^2}{p}}{(\g_i+n\mu)^2} + \frac{\s^2(p-1)}{np} \sum_{i=1}^n\left(\frac{\g_i}{\g_i+n\mu}\right)^2  \pt
  \end{split}
 \end{equation}
This decomposition has two direct consequences: 
\begin{itemize}
 \item the oracle risk of the multi-task procedure can be obtained by optimizing Equation~\eqref{optim.ind} independently over $\l$ and $\mu$;
 \item the estimator $\fh_{\AV}$ can be calibrated by independently calibrating two parameters.
\end{itemize}
It is now easy to optimize over the quantities in Equation~\eqref{optim.ind}. An interesting fact is that both sides have a natural and interesting interpretation, which we give now.

\subsection{Remark}

To avoid further ambiguities and to simplify the formulas we introduce the following notations for every $i\in\setn$:
\begin{equation*}
 \mu_i = \nu_i^1 = \frac{h_i^1+\dots + h_i^p}{\sqrt{p}}
\end{equation*}
and 
\begin{equation*}
 \varsigma_i^2 = \frac{\sum_{j=1}^p(h_i^j)^2}{p} -\left( \frac{\sum_{j=1}^p h_i^j}{p}\right)^2 = \frac{1}{p} \sum_{j=1}^p \paren{h_i^j-\frac{\sum_{j=1}^p h_i^j}{p}}^2 \virg
\end{equation*}
so that 
\begin{equation*}
 p\varsigma_i^2 = \sum_{j=2}^p(\nu_i^j)^2\pt
\end{equation*} 
\begin{rem}
 We can see that for every $i\in\setn$, $\mu_i/\sqrt{p}$ is the average of the $p$ target functions $f^j$, expressed on the basis diagonalizing $K$. Likewise, $\varsigma_i^2$ can be seen as the variance between the $p$ target functions $f^j$ (which does not come from the noise).
\end{rem}
\noindent Henceforth, the risk of $\fh_{\AV}(\l,\mu)$ over $(\l,\mu)$ is decoupled into two parts.  
\begin{itemize}
  \item With the parameter $\l$, a part which corresponds to the risk of a single-task ridge estimator, which regularizes the mean of the tasks functions, with a noise variance $\s^2/p$: 
\begin{equation} \label{risque.moy.lambda}
n\l^2\sum_{i=1}^n\frac{\frac{\mu_i^2}{p}}{(\g_i+n\l)^2} + \frac{\s^2}{np} \sum_{i=1}^n \left(\frac{\g_i}{\g_i+n\l}\right)^2  \pt
\end{equation}
  \item With the parameter $\mu$, a part which corresponds to the risk of a single-task ridge estimator, which regularizes the variance of the tasks functions, with a noise variance $(p-1)\s^2/p$: 
\begin{equation}\label{risque.var.lambda}
n\mu^2 \sum_{i=1}^n \frac{\varsigma_i^2}{(\g_i+n\mu)^2} + \frac{(p-1)\s^2}{np} \sum_{i=1}^n\left(\frac{\g_i}{\g_i+n\mu}\right)^2   \pt
\end{equation}
\end{itemize}

\begin{rem}
  Our analysis can also be used on any set of positive semi-definite matrices $\M$ that are jointly diagonalizable on an orthonormal basis, as was $\set{\MAV(\l,\mu),(\l,\mu)\in\R^2_+}$. 
  The element of interest then becomes the norms of the projections of the input tasks on the different eigenspaces (here, the mean and the resulting variance of the $p$ tasks). 
  An example of such a set is when the tasks are known to be split into several clusters, the assignement of each task to its cluster being known to the statistician. 
  The matrices that can be used then regularize the mean of the tasks and, for each cluster, the variance of the tasks belonging to this cluster.  
\end{rem}

\section{Precise analysis of the multi-task oracle risk \label{sec.MTrisk}}

In the latter section we showed that, in order to obtain the multi-task risk, we just had to optimize several functions, which have the form of the risk of a kernel ridge estimator. 
The risk of those estimators has already been widely studied. 
\citet{Johnstone1994} (see also the article of \citet{CapDeVito} for random design) showed that, for a single-task ridge estimator, if the coefficients of the decomposition of the input function on the eigenbasis of the kernel decrease as $i^{-2\d}$, with $2\d >1$, then the minimax rates for the estimation of this imput function is of order $n^{1/2\d-1}$. 
The kernel ridge estimator is then known to be minimax optimal, under certain regularity assumptions (see the work of \citet{bach2012_column_sampling} for more details). If the eigenvalues of the kernel are known to decrease as $i^{-2\b}$, then a single-task ridge estimator is minimax optimal under the following assumption:
\begin{equation} \tag{\bf\HypM}\label{hyp.minimax}
 1 < 2\d < 4\b+1 \pt
\end{equation}

The analysis carried in the former section shows that the key elements to express this risk are the components of the average of the signals ($\mu_i$) and the components of the variance of the signals ($\varsigma_i$) on the basis that diagonalises the kernel matrix $K$, together with the eigenvalues of this matrix ($\gamma_i$). 
It is then natural to impose the same natural assumptions that make the single-task ridge estimator optimal on those elements.
We first suppose that the eigenvalues of the kernel matrix have a polynomial decrease rate: 
\begin{equation} \tag{\bf\HK}\label{hyp.K}
 \forall i \in \setn,~\gamma_i = ni^{-2\b} \pt
\end{equation}
Then, we assume that the the components of the average of the signals and the variance of the signals also have a polynomial decrease rate:
\begin{equation} \tag{\bf\HAV}\label{hyp.AV}
 \forall i \in \setn,~\left\{ \begin{array}{rcl} \frac{\mu_i^2}{p} &=& C_1 n i^{-2\d}  \\ \varsigma_i^2 &=& C_2 n i^{-2\d} \end{array} \right. \pt
\end{equation}

\begin{rem}
  We assume for simplicity that both Assumptions~\eqref{hyp.K} and \eqref{hyp.AV} hold in equality, although the equivalence $\asymp$ is only needed. 
\end{rem}

\begin{ex}
  This example, related to Assumptions \eqref{hyp.AV}  and

  \noindent \eqref{hyp.K} by taking $\b=m$ and $2\d = k+2$, is detailed by \citet{Wah:1990} and by \citet{gu2002smoothing}. 
  Let $\mathcal{P}\paren{2\pi}$ the set of all square-integrable $2\pi$-periodic functions on $\R$, $m\in\NN^{\star}$ and define $\mathcal{H}=\set{f\in\mathcal{P}\paren{2\pi},~f_{|\croch{0,2\pi}}^{(m)}\in L^2\croch{0,2\pi}} $. 
  This set $\mathcal{H}$ has a RKHS structure, with a reproducing kernel having the Fourier base functions as eigenvectors. 
  The $i$-th eigenvalue of this kernel is $i^{-2m}$. 
  For any function $f\in\mathcal{P}\croch{0,2\pi}\cap\C^k\croch{0,2\pi}$, then its Fourier coefficient are $O(i^{-k})$. 
  For instance, if $f\in\mathcal{P}\croch{0,2\pi}$ such that $\forall x \in \croch{-\pi,\pi},~ f^{(k)}(x) = \absj{x}$, then its Fourier coefficients are $\asymp i^{-(k+2)}$.  
\end{ex}

Under Assumptions~\eqref{hyp.K} and \eqref{hyp.AV}, we can now more precisely express the risk of a multi-task estimator. Equation~\eqref{risque.moy.lambda} thus becomes 
\begin{align*} 
   & n\l^2\sum_{i=1}^n\frac{\frac{\mu_i^2}{p}}{(\g_i+n\l)^2} + \frac{\s^2}{np} \sum_{i=1}^n \left(\frac{\g_i}{\g_i+n\l}\right)^2 \\
  =~~ & n\l^2\sum_{i=1}^n\frac{C_1ni^{-2\d}}{(ni^{-2\b}+n\l)^2} + \frac{\s^2}{np} \sum_{i=1}^n \left(\frac{ni^{-2\b}}{ni^{-2\b}+n\l}\right)^2 \\
  =~~ & C_1\l^2 \sum_{i=1}^n\frac{i^{4\b-2\d}}{(1+\l i^{2\b})^2}+ \frac{\s^2}{np} \sum_{i=1}^n \frac{1}{\left(1+\l i^{2\b}\right)^2} \\
  =~~ &  R(n,p,\s^2,\l,\b,\d,C_1) \virg
\end{align*}
while Equation~\eqref{risque.var.lambda} becomes 
\begin{align*} 
  & n\mu^2 \sum_{i=1}^n \frac{\varsigma_i^2}{(\g_i+n\mu)^2} + \frac{(p-1)\s^2}{np} \sum_{i=1}^n\left(\frac{\g_i}{\g_i+n\mu}\right)^2 \\
=~~ &n\mu^2\sum_{i=1}^n\frac{C_2ni^{-2\d}}{(ni^{-2\b}+n\mu)^2} + \frac{(p-1)\s^2}{np} \sum_{i=1}^n \left(\frac{ni^{-2\b}}{ni^{-2\b}+n\mu}\right)^2 \\
=~~ & C_2\mu^2 \sum_{i=1}^n\frac{i^{4\b-2\d}}{(1+\mu i^{2\b})^2}+ \frac{(p-1)\s^2}{np} \sum_{i=1}^n \frac{1}{\left(1+\mu i^{2\b}\right)^2} \\
=~~ & R(n,p,(p-1)\s^2,\mu,\b,\d,C_2) \virg
\end{align*}
with 
\begin{equation}\label{eq.def.R}
 R(n,p,\s^2,x,\b,\d,C) = Cx^2 \sum_{i=1}^n\frac{i^{4\b-2\d}}{(1+x i^{2\b})^2}+ \frac{\s^2}{np} \sum_{i=1}^n \frac{1}{\left(1+x i^{2\b}\right)^2} \pt
\end{equation}

\begin{rem}
  It is to be noted that the function $R$ corresponds to the risk of a single-task ridge estimator when the decomposition of the input function on the eigenbasis of $K$ has $i^{-2\d}$ for coefficients and when $p=1$. 
  Thus, studying $R$ will allow us to derive both single-task and multi-task oracle rates. 
\end{rem}

\subsection{Study of the optimum of $R(n,p,\s^2,\cdot,\b,\d,C)$}

We just showed that the function $R(n,p,\s^2,\cdot,\b,\d,C)$ was suited to derive both single-task and multi-task oracle risk. 
\citet{bach2012_column_sampling} showed how to obtain a majoration on the function $R(n,p,\s^2,\cdot,\b,\d,C)$, so that its infimum was showed to match the minimax rates under Assumption~\eqref{hyp.minimax}. 

In this section, we first propose a slightly more precise upper bound of this risk function. 
We then show how to obtain a lower bound on this infimum that matches the aforementioned upper bound.
This will be done by precisely localizing the parameter minimizing $R(n,p,\s^2,\cdot,\b,\d,C)$. 

Let us first introduce the following notation:

\begin{definition}
  \begin{equation*}
    R^{\star}(n,p,\s^2,\b,\d,C) = \inf_{\l\in\R_+} \set{R(n,p,\s^2,\l,\b,\d,C)} \pt
  \end{equation*}
\end{definition}

We now give the upper bound on $R^{\star}(n,p,\s^2,\b,\d,C)$. For simplicity, we will denote by $\kbd$ a constant, defined in Equation~\eqref{def.kbd}, which only depends on $\b$ and $\d$. 

\begin{property}\label{prop.maj.risk}
Let $n$ and $p$ be positive integers, $\s$, $\b$ and $\d$ positive  real numbers such that \eqref{hyp.minimax}, \eqref{hyp.K} and \eqref{hyp.AV} hold. Then,
\begin{equation}\label{eq.maj.R.opt}
 R^{\star}(n,p,\s^2,\b,\d,C) \leq  \paren{2^{1/2\d}\paren{\frac{np}{\s^2}}^{1/2\d-1}C^{1/2\d}\kbd}\wedge \frac{\s^2}{p} \pt
\end{equation}
\end{property}
\begin{proof}
  Property~\ref{prop.maj.risk} is proved in Section~\ref{app.proof.prop.maj.risk} of the appendix.
\end{proof}

In the course of showing  Property~\ref{prop.maj.risk}, we obtained an upper bound on the risk function $R$ that holds uniformly on $\R_+$. 
Obtaining a similar (up to multiplicative constants) lower bound that also holds uniformly on $\R_+$ is unrealistic. 
However, we will be able to lower bound $R^{\star}$ by showing that $R$ is minimized by an optimal parameter $\l^{\star}$ that goes to 0 as $n$ goes to $+\infty$.

\begin{property}\label{prop.param.reg.maj}
 If Assumption~\eqref{hyp.minimax} holds,  the risk $R(n,p,\s^2,\cdot,\b,\d,C)$ attains its global minimum over $\R_+$ on $[0,\e\paren{\frac{np}{\s^2}}]$, with
 \begin{equation*}
  \e\paren{\frac{np}{\s^2}} = \sqrt{  C^{(1/2\d)-1}2^{1/2\d}\kbd} \times\frac{1}{\paren{\frac{np}{\s^2}}^{1/2-(1/4\d)}}\paren{1+\eta\paren{\frac{np}{\s^2}}} \virg
 \end{equation*}
 where $\eta(x)$ goes to $0$ as $x$ goes to $+\infty$.
\end{property}
\begin{proof}
  Property~\ref{prop.param.reg.maj} is shown in Section~\ref{proof.prop.param.reg.maj} of the appendix.
\end{proof}

\begin{rem}
 Thanks to the assumption made on $\d$, $\frac{1}{2\d}-1<0$ so that $\paren{\frac{np}{\s^2}}^{\frac{1}{2\d}-1}$ goes to 0 as $\frac{np}{\s^2}$ goes to $+\infty$. 
 This allows us to state that, if the other parameters are constant, $\l^{\star}$ goes to 0 as the quantity $\frac{np}{\s^2}$ goes to $+\infty$.
\end{rem}

We can now give a lower bound on $R^{\star}(n,p,\s^2,\b,\d,C)$. 
We will give two versions of this lower bound. First, we state a general result. 

\begin{property}\label{prop.min.risk}
 For every $(C,\b,\d)$ such that $1<2\d<4\b$ holds, there exits an integer $N$ and a constant $\a\in(0,1)$ such that, for every for every $(n,p,\s^2)$ verifying $\frac{np}{\s^2} \geq N$, we have 
 \begin{equation}\label{eq.min.R.opt}
  R^{\star}(n,p,\s^2,\b,\d,C) \geq  \paren{\a\paren{\frac{np}{\s^2}}^{1/2\d-1}C^{1/2\d}\kbd}\wedge\frac{\s^2}{4p} \pt
 \end{equation}
\end{property}

\begin{proof}
  Property~\ref{prop.min.risk} is proved in Section~\ref{sec.proof.prop.min.risk} of the appendix.
\end{proof}

\begin{rem}
 It is to be noted that $N$ and $\a$ only depend on $\b$ and $\d$. We can also remark that $\a$ can be taken arbitrarily close to 
 \begin{equation*}
  \frac{\int_{0}^{1} \frac{u^{\frac{1}{2\b}-1}}{(1+u)^2}du}{\int_{0}^{+\infty} \frac{u^{\frac{1}{2\b}-1}}{(1+u)^2}du} \wedge \frac{\int_{0}^{1} \frac{u^{\frac{1-2\d}{2\b}+1}}{(1+u)^2}du}{\int_{0}^{+\infty} \frac{u^{\frac{1-2\d}{2\b}+1}}{(1+u)^2}du} \pt
 \end{equation*}
 Numerical computations show that, by taking $\b=\d=2$, this constant is larger than $0.33$.
\end{rem}

\begin{rem}
  The assumption made on $\b$ and $\d$ is slighlty more restrictive than \eqref{hyp.minimax}, under which the upper bound is shown to hold and under which the single-task estimator is shown to be minimax optimal. 
\end{rem}

 We are now ensured that $R$ attains its global minimum on $\R_+$, thus we can give the following definition.

\begin{definition}
 For every $n$, $p$, $\s^2$, $\d$, $\b$ and $C$, under the assumption of Property~\ref{prop.param.reg.maj}, we introduce 
 \begin{equation*}
  \l^{\star}_R \in \argmin_{\l\in\R_+}\set{R(n,p,\s^2,\l,\b,\d,C)} \pt
 \end{equation*}
\end{definition}

 We now give a slightly refined version of Property~\ref{prop.min.risk}, by discussing whether this optimal parameter $\l^{\star}_R $ is larger or lower than the threshold $n^{-2\b}$. 
 This allows us to better understand the effect of regularizarion on the oracle risk $R^{\star}$. 

\begin{property}\label{prop.param.reg.min}
 For every $(\b,\d)$ such that $4\b>2\d>1$, integers $N_1$ and $N_2$ exist such that 
 \begin{enumerate}
  \item for every $(n,p,\s^2)$ verifying  $\frac{np}{\s^2} \geq N_1$ and $n^{1-2\d} \times \frac{p}{\s^2} \leq \frac{1}{N_2}$, then
    \begin{equation*}
      \l^{\star}_R \geq \frac{1}{n^{2\b}}
    \end{equation*}
    and
  \begin{equation*}
     R^{\star}(n,p,\s^2,\b,\d,C) \asymp \paren{\frac{\s^2}{np}}^{1-1/2\d} \pt
   \end{equation*}
 \item for every $(n,p,\s^2)$ verifying  $\frac{np}{\s^2} \geq N_1$ and $n^{1-2\d}  \times \frac{p}{\s^2} \geq N_2$, then
    \begin{equation*}
      \l^{\star}_R \leq \frac{1}{n^{2\b}} 
    \end{equation*}
    and 
   \begin{equation*}
     R^{\star}(n,p,\s^2,\b,\d,C) \asymp R(n,p,\s^2,0,\b,\d,C) \asymp \frac{\s^2}{p} \ptvirg
   \end{equation*}
 \end{enumerate}

\end{property}
 
\begin{proof}
  Property~\ref{prop.param.reg.min} is proved in Section~\ref{sec.proof.prop.param.reg.min} of the appendix.
\end{proof}

\begin{rem}
 If $p \leq n\s^2$ and $\d > 1$ then we are in the first case, for a large enough $n$. 
This is a case where regularization has to be employed in order to obtain optimal convergence rates. 
\end{rem}
\begin{rem}
 If $\s^2$ and $n$ are fixed and $p$ goes to $+\infty$ then we are in the second case. 
 It is then useless to regularize the risk, since the risk can only be lowered by a factor $4$. 
 This also corresponds to a single-task setting where the noise variance $\s^2$ is very small and where the estimation problem becomes trivial. 
\end{rem}

\subsection{Multi-task oracle risk}
We can now use the upper and lower bounds on $R^{\star}$ to control the oracle risk of the multi-task estimator. 
We define 
\begin{equation*}
 \l^{\star} \in \argmin_{\l\in\R_+}\set{ R(n,p,\s^2,\l,\b,\d,C_1)}
\end{equation*}
and
\begin{equation*}
 \mu^{\star} \in \argmin_{\mu\in\R_+}\set{R(n,p,(p-1)\s^2,\mu,\b,\d,C_2)} \pt
\end{equation*}
Property~\ref{prop.param.reg.maj} ensures that $\l^{\star}$ and $\mu^{\star}$ exist, even though they are not necessarily unique. The oracle risk then is 
\begin{equation*}
 \mathfrak{R}^{\star}_{\MT} =  \inf_{(\l,\mu)\in\R_+^2}\set{\frac{1}{np}\esp{\norr{\fh_{\MAV(\l,\mu)} -f}^2}} = \frac{1}{np}\esp{\norr{\fh_{\MAV(\l^{\star},\mu^{\star})} -f}^2} \pt
\end{equation*}
We now state the main result of this paper, which simply comes from the analysis of $R^{\star}$ performed above. 
\begin{theorem}\label{thm.MT.oracle}
 For every $n$, $p$, $C_1$, $C_2$, $\s^2$, $\b$ and $\d$ such that Assumption~\eqref{hyp.minimax} holds, we have 
\begin{equation}\label{eq.maj.oracle.MT}
  \mathfrak{R}^{\star}_{\MT} \leq2^{1/2\d} \paren{\frac{np}{\s^2}}^{1/2\d-1}\kbd \croch{C_1^{1/2\d} + (p-1)^{1-(1/2\d)}C_2^{1/2\d}} \pt
\end{equation}
Furthermore, constants $N$ and $\a \in (0,1)$ exist such that, if $n \geq N$, $p/\s^2 \leq n$ and $2<2\d<4\b$, we have 
\begin{equation}\label{eq.min.oracle.MT}
  \mathfrak{R}^{\star}_{\MT} \geq \a\paren{\frac{np}{\s^2}}^{1/2\d-1}\kbd \croch{C_1^{1/2\d} + (p-1)^{1-(1/2\d)}C_2^{1/2\d}} \pt
\end{equation}
\end{theorem}

\begin{proof}
 The risk of the multi-task estimator $\fh_{\MAV(\l,\mu)}$ can be written as 
 \begin{equation*}
   R(n,p,\s^2,\l,\b,\d,C_1) + R(n,p,(p-1)\s^2,\mu,\b,\d,C_2) \pt
 \end{equation*}
 We then apply Properties~\ref{prop.maj.risk} and \ref{prop.min.risk}, since $p/\s^2 \leq n$ implies that $p/(p-1)\s^2 \leq n$.
 The assumption $\d >1$ ensures that the first setting of Property~\ref{prop.param.reg.min} holds. 
\end{proof}

\begin{rem}
  An interesting fact is that the oracle multi-task risk is of the order $(np/\s^2)^{1/2\d-1}$. 
  This corresponds to the risk of a single-task ridge estimator with sample size $np$. 
\end{rem}

\begin{rem}
  As noted before, the assumption under which the lower bound holds is slightly stronger than Assumption~\eqref{hyp.minimax}.
\end{rem}

\section{Single-task oracle risk \label{sec.STrisk}}

In the former section we obtained a precise approximation of the multi-task oracle risk $\mathfrak{R}^{\star}_{\MT}$. 
We would now like to obtain a similar approximation for the single-task oracle risk $\mathfrak{R}^{\star}_{\ST}$. 
In the light of Section~\ref{sec.risk}, the only element we need to obtain the oracle risk of task $j\in\setp$ is the expression of $(h_i^j)_{i=1}^n$, that is, the coordinates of $(f^j(X_i))_{i=1}^n$ on the eigenbasis of $K$.
Unfortunately, Assumption~\eqref{hyp.AV} does not correspond to one set of task functions $(f^1,\dots,f^p)$. 
Thus, since several single-task settings can lead to the same multi-task oracle risk, we now explicitly define two repartitions of the task functions $(f^1,\dots,f^p)$, for which the single-task oracle risk will be computed. 

\begin{itemize}
\item ``2 points'': suppose, for simplicity, that $p$ is even and that  
\begin{equation} \label{2Points} \tag{2Points}
 f^1  = \dots = f^{p/2} \et f^{p/2+1} = \dots = f^p\pt 
\end{equation}
\item ``1 outlier'': 
\begin{equation} \label{1Out} \tag{1Out}
 f^1 = \dots = f^{p-1} \pt
\end{equation}
\end{itemize}

Both assumptions correspond to settings in which the multi-task procedure would legitimately be used. 
Assumption~\eqref{2Points} models the fact that all the functions lie in a cluster of small radius. 
It supposes that the functions are split into two groups of equal size, in order to be able to explicitly derive the single-task oracle risk.
Assumption~\eqref{1Out} supposes that all the functions are grouped in one cluster, with one outlier.
In order to make the calculations possible, all the functions in one group are assumed to be equal. 
Since this is not a fully convincing situation to study the behaviour of the multi-task oracle, simulation experiments were also run on less restrictive settings.
The results of those experiments are shown in Section~\ref{sec.simus}. 

\begin{rem}
 The hypotheses \eqref{2Points} and \eqref{1Out} made on the functions $f^j$ can be expressed on $(h_i^j)$. Assumption~\eqref{2Points} becomes 
\begin{equation*} 
\forall i \in \setn,~ h^1_i = \dots = h^{p/2}_i \et h^{p/2+1}_i = \dots = h^p_i \virg
\end{equation*}
 while Assumption ~\eqref{1Out} becomes
\begin{equation*} 
\forall i \in \setn,~h^1_i = \dots = h^{p-1}_i \pt
\end{equation*}
\end{rem}
Under those hypotheses we now want to derive an expression of $(h_i^1,\dots,h_i^p)$ given $(\mu_i,\varsigma_i)$ so that we can exactly compute the single-task oracle risk. 
Remember we defined for every $i \in \{1,\dots,n\}$,
\begin{equation*}
 \mu_i = \frac{1}{\sqrt{p}} \sum_{j=1}^p h_i^j 
\end{equation*}
and
\begin{equation*}
 \varsigma_i^2 = \frac{1}{p} \sum_{j=1}^p \paren{h_i^j}^2 - \frac{\mu_i^2}{p} = \frac{1}{p} \sum_{j=1}^p\paren{h_i^j - \frac{\mu_i}{\sqrt{p}}}^2 \pt
\end{equation*}
We also re-introduce the single-task oracle risk:
\begin{equation*}
 \mathfrak{R}^{\star}_{\ST} = \inf_{(\l^1,\dots,\l^p)\in\R_+^p}\set{\frac{1}{np}\sum_{j=1}^p \norr{\fh^j_{\l^j}-f^j}^2} \pt	
\end{equation*}
We now want to closely study this single-task oracle risk, in both settings.

\subsection{Analysis of the oracle single-task risk for the ``2 points'' case \eqref{2Points}}
In this section we write the single-task oracle risk when Assumption~\eqref{2Points} holds. 
As shown in Lemma~\ref{lemma.risk.2Points}, the risk of the estimator $\fh^j_{\l}=A_{\l}y^j$ for the $j$th task, which we denote by $R^j(\l)$, verifies 
 \begin{equation*}
  R(n,1,\s^2,\l,\b,\d,\paren{\sqrt{C_1}-\sqrt{C_2}}^2) \leq R^j(\l) \leq  R(n,1,\s^2,\l,\b,\d,\paren{\sqrt{C_1}+\sqrt{C_2}}^2) \pt
 \end{equation*}
Both upper and lower parts eventually behave similarly. 
In order to simplify notations and to avoid having to constantly write two risks, we will assume that half of the tasks have a risk equal to the right-hand side of the later inequality and the other half a risk equal to the left-hand side of this inequality.
This leads to the following assumption:
\begin{equation} \tag{\bf H$_{\textrm{2Points}}$}\label{hyp.2Points}
 \forall i \in \setn,~ \left\{ \begin{array}{rcl} h_i^1 &=& \sqrt{n}i^{-\d}(\sqrt{C_1}+\sqrt{C_2}) \\ h_i^p &=& \sqrt{n}i^{-\d}(\sqrt{C_1}-\sqrt{C_2}) \end{array} \right. \pt
\end{equation}
This minor change does not affect the convergence rates of the estimator. 
Consequently, if $1 \leq j \leq p/2$ the risk for task $j$ is $R(n,1,\s^2,\l,\b,\d,\paren{\sqrt{C_1}+\sqrt{C_2}}^2)$ so that the oracle risk for task $j$ is, given that $n\s^2 \geq 1$, 
\begin{equation*}
\asymp \paren{\frac{n}{\s^2}}^{1/2\d-1}\kbd\times \paren{\sqrt{C_1}+\sqrt{C_2}}^{1/\d}\virg
\end{equation*}
and if $p/2+1\leq j \leq p$ the risk for task $j$ is $R(n,1,\s^2,\l,\b,\d,\paren{\sqrt{C_1}-\sqrt{C_2}}^2)$ so that the oracle risk for task $j$ is, given that $n\s^2 \geq 1$,
\begin{equation*} 
\asymp \paren{\frac{n}{\s^2}}^{1/2\d-1}\kbd\times \absj{\sqrt{C_1}-\sqrt{C_2}}^{1/\d}\virg
\end{equation*}

\begin{rem}
 We can remark that \eqref{hyp.2Points} implies \eqref{2Points} and that \eqref{hyp.2Points} implies \eqref{hyp.AV}, as shown in Lemma \ref{lemma.hyp.2Points}. Consequently, if \eqref{hyp.2Points} holds, we have, for every $i\in\setn$, $h_i^1 = \frac{\mu_i}{\sqrt{p}} + \varsigma_i$ and $h_i^p = \frac{\mu_i}{\sqrt{p}} - \varsigma_i$.
\end{rem}

\begin{cor}
  For every $n$, $p$, $C_1$, $C_2$, $\s^2$, $\b$ and $\d$ such that $2<2\d<4\b$ and $n\s^2>1$ and that Assumptions~\eqref{hyp.2Points} and \eqref{hyp.K} hold, then 
  \begin{equation} \label{ST_risk_2Points}
 \mathfrak{R}^{\star}_{\ST}  \asymp \paren{\frac{np}{\s^2}}^{1/2\d-1}\frac{\kbd}{2} \times p^{1-1/2\d}   \croch{\paren{\sqrt{C_1}+\sqrt{C_2}}^{1/\d}+\absj{\sqrt{C_1}-\sqrt{C_2}}^{1/\d}}\pt
\end{equation}
\end{cor}

\subsection{Analysis of the oracle single-task risk for the ``1 outlier'' case \eqref{1Out}}
In this section we suppose that Assumption~\eqref{1Out} holds.
As shown in Lemma~\ref{lemma.risk.1Out}, we can lower and upper bound the risks of the single-tasks estimators by functions of the shape $R(n,p,\s^2,\l,\b,\d,C)$. 
As in the latter section, to avoid the burden of writing two long risk terms at every step, and since all those risks have the same convergence rates, we suppose from now on the new assumption: 
\begin{equation}\label{hyp.1Out}\tag{\bf H$_{\textrm{1Out}}$}
 \forall i \in \setn \left \{ \begin{array}{rcl} h_i^1 &=& \sqrt{n}i^{-\d}\paren{\sqrt{C_1} + \frac{1}{\sqrt{p-1}}\sqrt{C_2}} \\
 h_i^p &=&  \sqrt{n}i^{-\d}\paren{\sqrt{C_1} -\sqrt{p-1}\sqrt{C_2}} \end{array} \right.\pt
\end{equation}
This minor change does not affect the convergence rates of the estimator. 
Consequently, if $1 \leq j \leq p-1$ the risk for task $j$ is $R(n,1,\s^2,\l,\b,\d,\paren{\sqrt{C_1} + \sqrt{\frac{C_2}{p-1}}}^2)$ so that the oracle risk for task $j$ is, given that $n\s^2 \geq 1$, 
\begin{equation*} 
\asymp \paren{\frac{n}{\s^2}}^{1/2\d-1}\kbd \times \paren{\sqrt{C_1} + \sqrt{\frac{C_2}{p-1}}}^{1/\d}\virg
\end{equation*}
while the risk for task $p$ is $R(n,1,\s^2,\l,\b,\d,\paren{\sqrt{C_1} -\sqrt{(p-1)C_2}}^2)$ so that the oracle risk for task $p$ is, given that $n\s^2 \geq 1$, 
\begin{equation*} 
  \asymp \paren{\frac{n}{\s^2}}^{1/2\d-1}\kbd\times \absj{\sqrt{C_1} -\sqrt{(p-1)C_2}}^{1/\d}\pt
\end{equation*}

\begin{rem}
 We can also remark here that \eqref{hyp.1Out} implies \eqref{1Out} and that \eqref{hyp.1Out} implies \eqref{hyp.AV}, as shown in Lemma~\ref{lemma.risk.1Out}. Consequently, if \eqref{hyp.1Out} holds, we have, for every $i\in\setn$, $h_i^1 = \frac{\mu_i}{\sqrt{p}} + \frac{1}{\sqrt{p-1}}\varsigma_i$ and $h_i^p = \frac{\mu_i}{\sqrt{p}} - \sqrt{p-1}\varsigma_i$.
\end{rem}oracle

\begin{cor}
  For every $n$, $p$, $C_1$, $C_2$, $\s^2$, $\b$ and $\d$ such that $2<2\d<4\b$ and $n\s^2>1$ and that Assumptions~\eqref{hyp.1Out} and \eqref{hyp.K} hold, then 
  \begin{multline} \label{ST_risk_1Out}
    \mathfrak{R}^{\star}_{\ST}  \asymp \paren{\frac{np}{\s^2}}^{1/2\d-1}\kbd\\ \times p^{1-1/2\d}\croch{\frac{p-1}{p}\paren{\sqrt{C_1} + \sqrt{\frac{C_2}{p-1}}}^{1/\d} +\frac{1}{p}\absj{\sqrt{C_1} -\sqrt{(p-1)C_2}}^{1/\d}}\pt
  \end{multline}
\end{cor}

\section{Comparison of  multi-task  and single-task \label{sec.MTSTcomp}}
In the two latter section we obtained precise approximations of the multi-task oracle risk, $\mathfrak{R}^{\star}_{\MT}$, and of the single-task oracle risk, $\mathfrak{R}^{\star}_{\ST}$, under either Assumption~\eqref{hyp.2Points} or \eqref{hyp.1Out}. 
We can now compare both risks in either setting, by studying their ratio 
\begin{equation*}
  \rho =  \frac{\mathfrak{R}^{\star}_{\MT}}{\mathfrak{R}^{\star}_{\ST}} \pt
\end{equation*}
We will express the quantity $\rho$ as a factor of 
\begin{equation*}
  r = \frac{C_2}{C_1} \pt
\end{equation*}
The parameter $r$ controls the amount of the signal which is contained in the mean of the functions.
When $r$ is small, the mean of the tasks contains much more signal than the variance of the tasks, so that the tasks should be ``similar''. 
This is a case where the multi-task oracle is expected to perform better than the single-task oracle. 
On the contrary, when $r$ is large, the variance of the tasks is more important than the mean of the tasks. 
This is a case where the tasks would be described as ``non-similar''. 
It is then harder to conjecture whether the single-task oracle performs better than the multi-task oracle and, as we will see later, the answer to this greatly depends on the setting. 

\subsection{Analysis of the oracle multi-task improvement for the ``2 points'' case \eqref{2Points}}
We now express $\rho$ as a function of $r$ when the tasks are split in two groups. 

\begin{cor}\label{cor.rho.2Points}
  For every $n$, $p$, $C_1$, $C_2$, $\s^2$, $\b$ and $\d$ such that $2<2\d<4\b$ and $n\s^2>p$ and that Assumptions~\eqref{hyp.2Points} and \eqref{hyp.K} hold, then 
  \begin{equation} \label{rho_2Points}
    \rho \asymp \frac{p^{1/2\d-1}+(\frac{p-1}{p})^{1-(1/2\d)}r^{1/2\d}}{\paren{1+\sqrt{r}}^{1/\d}+\absj{1 -\sqrt{r}}^{1/\d}} \pt
  \end{equation}
\end{cor}

\begin{rem}
  The right-hand side of Equation~\eqref{rho_2Points} is always smaller than $\frac{1}{2}$. 
  Thus, under the assumptions of Corollary~\ref{cor.rho.2Points}, the multi-task oracle risk can never be arbitrarily worse than the single-task oracle risk.
\end{rem}

We can first see that,  under the assumptions of Corollary~\ref{cor.rho.2Points}, $\rho = \Theta\paren{p^{1/2\d-1}}$ as $r$ goes to 0. 
This is the same improvement that we get we multiplying the sample-size by $p$. 
We also have $\rho = \Theta\paren{\paren{\frac{p-1}{p}}^{1-(1/2\d)}}$ as $r$ goes to $+\infty$, so that the multi-task oracle and the single-task oracle behave similarly. 
 Finally, $\rho =\Theta\paren{\frac{r^{1/2\d}}{\paren{1+\sqrt{r}}^{1/\d}+\absj{1 -\sqrt{r}}^{1/\d}}}$ as $p$ goes to $+\infty$, so that the behaviours we just discussed are still valid with a large number of tasks. 

\subsection{Analysis of the oracle multi-task improvement for the ``1 outlier'' case \eqref{1Out} \label{sub.sec.an1Out}}
We now express $\rho$ as a function of $r$ when the tasks are grouped in one group, with one outlier. 

\begin{cor}\label{cor.rho.1Out}
  For every $n$, $p$, $C_1$, $C_2$, $\s^2$, $\b$ and $\d$ such that $2<2\d<4\b$ and $n\s^2>p$ and that Assumptions~\eqref{hyp.1Out} and \eqref{hyp.K} hold, then 
\begin{equation}\label{rho_1Out}
 \rho \asymp\frac{p^{1/2\d-1}+\paren{\frac{p-1}{p}}^{1-(1/2\d)}r^{1/2\d}}{\frac{p-1}{p}\paren{1+\sqrt{\frac{r}{p-1}}}^{1/\d} +\frac{1}{p}\absj{1-\sqrt{r(p-1)}}^{1/\d}} \pt
\end{equation}
\end{cor}

We can see that,  under the assumptions of Corollary~\ref{cor.rho.1Out}, $\rho = \Theta\paren{p^{1/2\d-1}}$ as $r$ goes to 0. 
As in the latter section, this is the same improvement that we get we multiplying the sample-size by $p$. 
However, $\rho = \Theta\paren{\paren{\frac{p-1}{p}}^{1-1/2\d} \times \frac{p(p-1)^{-1/2\d}}{1+(p-1)^{1-1/\d}}}$ as $r$ goes to $+\infty$. 
This quantity goes to $+\infty$ as $p\longrightarrow +\infty$, so that the multi-task oracle performs arbitrarily worse than the single-task one in this asymptotic setting. 
Finally, $\rho=\Theta\paren{r^{1/2\d}}$ as $p$ goes to $+\infty$. 
This quantity goes to $+\infty$ as $r$ goes to $+\infty$, so that the behaviours we just mentioned stay valid with a large number of tasks.

\subsection{Discussion \label{subsec.diff}}

  When $r$ is small, either under Assumption~\eqref{2Points} or \eqref{1Out}, the mean of the signal is much stronger than the variance. 
  Thus, the multi-task procedure performs better than the single-task one.
  \begin{ex}
    If $r=0$, then all the tasks are equal. The improvement of the multi-task procedure over the single-task one then is $p^{1/2\d-1}$. This was expected: it corresponds to the risk of a ridge regression with a $np$-sample.  
  \end{ex}
  As $r$ goes to $0$, the multi-task oracle outperforms its single-task counterpart by a factor $p^{1/2\d-1}$. When $p$ is large (but, remember, this only holds when $p/\s^2 \leq n$, so $n$ also has to be large), this leads to a substantial improvement. It is easily seen that, for any constant $C>1$, if $r \leq (C-1)^{2\d} (p-1)^{1-2\d}$, then the right-hand side of Equation~\eqref{rho_2Points} becomes smaller than $Cp^{1/2\d-1}$. Thus, if the tasks are similar enough, the multi-task oracle performs as well as the oracle for a $np$-sample, up to a constant.  

\medskip

  On the contrary, when $r$ is large, the variance carries most of the signal, so that the tasks differ one from another. As $r$ goes to $+\infty$, the two settings have different behaviours:
    \begin{itemize}
     \item under Assumption~\eqref{2Points} (that is, when we are faced to two equally-sized groups), the oracle risks of the multi-task and of the single-task estimators are of the same order: they can only differ by a multiplicative constant;
     \item under Assumption~\eqref{1Out} (that is, when we are faced to one cluster and one outlier), the single-task oracle outperforms the multi-task one, by a factor which is approximatly $p^{1/\d}$.  
    \end{itemize}
  
  Finally, Assumption~\eqref{2Points} presents no drawback for the multi-task oracle, since under those hypotheses its performance cannot be worse than the single-task oracle's one. 
  On the contrary, Assumption~\eqref{1Out} presents a case where the use of a multi-task technique greatly increases the oracle risk, when the variance between the tasks is important, while it gives an advantage to the multi-task oracle when this variance is small. 
  The location where the multi-task improvement stops  corresponds to the barrier $\rho = 1$. 
  Studying this object seems difficult, since we only know $\rho$ up to a multiplicative constant.  
  Also, finding the contour lines of the righ-hand side of Equation~\eqref{rho_1Out} does not seem to be an easy task. 
  In Section~\ref{sec.simus}, we will run simulations in situations where the oracle risk can no longer be explicitly derived. 
  We will show that the behaviours found in these two examples still appear in the simulated examples.

\section{Risk of a multi-task estimator \label{sec.riskest}}

 \citet{solnon:hal-00610534} introduced an entirely data-driven estimator to calibrate $\MAV(\l,\mu)$ over $\R_+^2$. 
One of their main results is an oracle inequality, that compares the risk of this estimator to the oracle risk. 
Thus, $\mathfrak{R}^{\star}_{\MT}$ is attainable by a fully data-driven estimator. 
We now show that our estimation of the multi-task oracle risk is precise enough so that we can use it in the mentionned oracle inequality and still have a lower risk than the single-task oracle one.

The following assumption will be used, with $\df(\l) = \tr(A_{\l})$ and $A_{\l} = K(K+n\l I_n)^{-1}$ : 
\begin{equation}\tag{Hdf}\label{Hdf}
  \left.
    \begin{aligned}
      & \forall j \in \set{1, \ldots, p} , \, \exists \l_{0,j} \in (0,+\infty) \, ,  \\
      & \qquad \df(\l_{0,j}) \leq \sqrt{n} \quad \mbox{and} \quad \frac{1}{n} \norm{ (A_{\l_{0,j}}-I_n) f^j }_2^2 \leq \s^2 \sqrt{\frac{\ln n}{n}}
    \end{aligned}
    \enspace \right\}
\end{equation}

We will also denote $\M = \set{\MAV(\l,\mu),~ (\l,\mu)\in\R_+^2}$ and $\Mh_{\textrm{HM}}$ the estimator introduced in \citet{solnon:hal-00610534}, which belongs to $\M$. Theorem 29 of \citet{solnon:hal-00610534} thus states:

 \begin{theorem} \label{thm.oracle.HM}
Let $\a = 2$, $\th \geq 2$, $p\in\NN^{\star}$ and assume~\eqref{Hdf} holds true. An absolute constant $L>0$ and a constant $n_1(\th)$ exist such that the following holds as soon as $n \geq n_1(\th)$.
\begin{equation} \label{resultatoracle.esp.HM}
  \begin{split}
  \esp{\frac{1}{np}\norr{\fh_{\Mh_{\textrm{HM}}}-f}^2} \leq \left( 1+\frac{1}{\ln(n)} \right)^2 \esp{\inf_{M \in \M}\left\{\frac{1}{np} \norr{ \fh_{M}-f}^2  \right\}} \\+L\s^2 (2 + \th)^2p\frac{\ln(n)^3}{n}
 + \frac{p}{n^{\th/2}} \frac{\norr{f}^2  }{n p}
   \pt
  \end{split}
\end{equation}
\end{theorem}
We first remark that 
\begin{equation*}
 \esp{\inf_{M \in \M}\left\{\frac{1}{np} \norr{ \fh_{M}-f}^2  \right\}} \leq  \mathfrak{R}^{\star}_{\MT}\pt
\end{equation*}
We can now plug the oracle risk in the oracle inequality~\eqref{resultatoracle.esp.HM}. 
Then, if we suppose that, for $i\in\setn$ and $j\in\setp$, $(h_i^j)^2 = n C^j i^{-2\d}$, we have that 
\begin{equation*}
 \norr{f}^2 = \sum_{j=1}^p\sum_{i=1}^n (h_i^j)^2 = n\sum_{j=1}^p C^j \sum_{i=1}^n i^{-2\d} \leq n\zeta{(2\d)}\sum_{j=1}^p C^j \pt
\end{equation*}
  
\begin{rem}
 Assumption \eqref{2Points} means that for every $i\in\setn$, if $1\leq j\leq p/2$, 
\begin{equation*}
 C^j = \paren{\sqrt{C_1}+\sqrt{C_2}}^2
\end{equation*}
 and if $p/2+1\leq j\leq p$, 
\begin{equation*}
 C^j = \paren{\sqrt{C_1}-\sqrt{C_2}}^2 \pt
\end{equation*}
 Assumption $\eqref{1Out}$ means that for every $i\in\setn$, if $1\leq j \leq p-1$, 
\begin{equation*}
 C^j=\paren{\sqrt{C_1} + \sqrt{\frac{C_2}{p-1}}}^2
\end{equation*}
while 
\begin{equation*}
 C^p = \paren{\sqrt{C_	1} -\sqrt{(p-1)C_2}}^2 \pt
\end{equation*} 
\end{rem}

\begin{property}
 Under Assumptions \eqref{hyp.K} and \eqref{hyp.AV} with $2\d>2$, there exists a constant $N_1$ such that for every $n\geq N_1$, Assumption \eqref{Hdf} holds. 
\end{property}

\begin{proof}
 We can see that Assumption \eqref{Hdf} is made independently on every task. Thus we can suppose that $p=1$. Let us denote $b(\l) = n^{-1} \norr{(A_\l-I_n)f}^2$. We can see that if there exists constants $c>0$ and $d>1$ such that for every $\l\in\R_+$ $b(\l) \leq c\s^2\df(\l)^{-d}$, then Assumption \eqref{Hdf} holds for $n$ large enough. Indeed, let $\l \in \R_+$ such that $\df(\l) \leq \sqrt{n}$. Then, if $b(\l) \leq c\s^2\df(\l)^{-d}$, $b(\l) \leq \s^2 c (\sqrt{n})^{-d} \leq \s^2 c \frac{n^{(-d+1)/2}}{\sqrt{n}}$. It just suffices to see that, for $n$ large enough, $cn^{-d+1} \leq \ln(n)$.

Using Lemmas \ref{lemma.S1} and \ref{lemma.S2} we can see that, for every $\l\in\R_+$,
\begin{equation*}
 b(\l) \leq \frac{\l^{\frac{2\d-1}{2\b}}}{\b}I_1(\b,\d)
\end{equation*}
 and, for $n$ large enough, there exists a constant $\a$ such that, for every $\l\in\R_+$,
 \begin{equation*}
  \df(\l) = \tr{A_{\l}} \geq \a\frac{\l^{\frac{-1}{2\b}}}{2\b}I_2(\b)
 \end{equation*}
 Thus, for $n$ large enough, there exists a constant $c$ (depending on $\s^2$, $\b$ and $\d$) such that, for every $\l\in\R_+$,
 \begin{equation*}
  b(\l) \leq c\s^2 \tr(A_{\l})^{-(2\d-1)} \pt
 \end{equation*}
 Hence, if $2\d>2$, there exists a constant $N_1$ such that for every $n\geq N_1$, Assumption \eqref{Hdf} holds. 
\end{proof}

Thus, we can apply Theorem~\ref{thm.oracle.HM} to the estimator $\fh_{\Mh_{\textrm{HM}}}$ under either Assumption \eqref{2Points} or \eqref{1Out} (and we denote by $\rho$ either $\rho_{2Points}$ or $\rho_{1Out}$).

\begin{property}
 For every positive numbers $(\b,\d,\th,C_1,C_2)$ verifying $4\b>2\d>2$ and $\th>1$,  there exists positive constants $(N(\b,\d,\th),L)$ such that, for every $(n,p,\s^2)$ verifying $n\geq N$ and $\frac{p}{\s^2} \leq n$, if Assumption \eqref{hyp.K} and if either Assumption~\eqref{hyp.2Points} or Assumption~\eqref{hyp.1Out} hold, the ratio between the risk of the estimator  $\fh_{\Mh_{\textrm{HM}}}$ and the single-task oracle risk verifies 
 \begin{equation*}
  \begin{split}
    \frac{ \esp{\frac{1}{np}\norr{\fh_{\Mh_{\textrm{HM}}}-f}^2}}{\mathfrak{R}^{\star}_{\ST}} \leq \left( 1+\frac{1}{\ln(n)} \right)^2\rho ~~~~~~~~~~~~~~~~~~~~~~~~~~~~~~~~~~~~~~~~~~~~~~  \\ ~~~~~~~~~~~~~~~~~~~~~~~~~~~~~~~~+ Cst \times \frac{L\s^2(2 +\th)^2 p\frac{\ln(n)^3}{n} + \frac{p\zeta(2\d)}{n^{\th/2}}\frac{1}{p}\sum_{j=1}^pC^j}{\paren{\frac{n}{\s^2}}^{1/2\d-1}\kbd\times\frac{1}{p}\sum_{j=1}^p (C^j)^{1/2\d}} \pt
  \end{split}
  \end{equation*}
\end{property}
 
\begin{proof}
 This is a straightforward application of the preceding results. 
\end{proof}

We now show that the latter fully data-driven multi-task ridge estimator achieves a lower risk than the single-task ridge oracle, in both settings \eqref{2Points} and \eqref{1Out}. 
\begin{cor}\label{cor.est}
  For every positive numbers $(\b,\d,\th,\s^2,\e)$ verifying $4\b>2\d>2$ and $\th>2$, there exists positive constants $(N,r)$ such that, for every $(n,p,C_1,C_2)$ verifying $n\geq N$, $\frac{p}{\s^2}\leq n^{1/4\d}$ and $\frac{C_2}{C_1} \leq r$, if Assumptions \eqref{hyp.K} holds and if either Assumption~\eqref{hyp.2Points} or Assumption~\eqref{hyp.1Out} hold, the ratio between the risk of the estimator  $\fh_{\Mh_{\textrm{HM}}}$ and the single-task oracle risk verifies 
\begin{equation*}
  \frac{ \esp{\frac{1}{np}\norr{\fh_{\Mh_{\textrm{HM}}}-f}^2}}{\mathfrak{R}^{\star}_{\ST}} < \e \pt
 \end{equation*}
\end{cor}

\begin{proof}
 First, we can see that under either Assumption~\eqref{2Points} or Assumption~\eqref{1Out}, both $\frac{1}{p}\sum_{j=1}^pC^j$ and $\frac{1}{p}\sum_{j=1}^p (C^j)^{1/2\d}$ converge, as $p$ goes to $+\infty$, to quantities only depending on $C_1$, $C_2$ and $\d$ and are thus bounded with respect to $p$. 
Then, as it was shown in the previous section, both $\rho_{2Points}$ and $\rho_{1Out}$ go to $0$ as $\frac{C_1}{C_2}$ goes to $0$. Finally, we can see that $\frac{p}{\s^2}\leq n^{1/4\d}$ implies that $\frac{p}{\s^2}\leq n$ and that 
  \begin{equation*}
   \frac{\s^2 p\frac{\ln(n)^3}{n}}{\paren{\frac{n}{\s^2}}^{1/2\d-1}} = \paren{\s^2}^{1/2\d}\times p \times \frac{\ln(n)^3}{n^{1/2\d}} \leq \paren{\s^2}^{1+1/2\d} \times \frac{\ln(n)^3}{n^{1/4\d}} \converge_{n\to +\infty}0 
  \end{equation*}
 together with
 \begin{equation*}
  \frac{\frac{p}{n^{\th/2}}}{\paren{\frac{n}{\s^2}}^{1/2\d-1}} \leq  \paren{\s^2}^{1/2\d} \times n^{1-\th/2-1/4\d}\converge_{n\to +\infty}0 \pt
 \end{equation*}
\end{proof}

\begin{rem}
  The result shown in Corollary~\eqref{cor.est}  establishes that a fully data-driven multi-task estimator outperforms an oracle single-task estimator, which is minimax optimal . 
\end{rem}

\section{Numerical experiments \label{sec.simus}}

The hypotheses we used in the former sections, although sufficient to precisely derive the risk of the estimator, do not reflect realistic situations. 
In this section we study less restrictive settings. 
However, we are no longer able to obtain simple formulas for the oracle risk as we did before. 
Thus, we resort to numerical simulations to illustrate the behaviour of both single-task and multi-task oracles. 

\subsection{Setting A: relaxation of Assumptions~\eqref{hyp.AV} and \eqref{2Points} in order to get one general group of tasks}
In the latter sections we modeled the fact that the $p$ target functions are close. 
However, due to technical constraints we were only able to deal with cases where the functions are split into two groups and are then equal inside each group, thus introducing Assumptions~\eqref{2Points} and \eqref{1Out}. 
We propose here to extend this setting by simulating a more general group of tasks.
Those tasks should all be at a comparable distance from a centroid function. 

We suppose that  $(\e_i^j)_{i\in\setn,j\in\setp}$ is  a sequence of i.i.d. random variables, independent of $(X_i)_{i\in\setn}$, following a Rademacher distribution (that is, such that $\Proba(\e_i^j=1) = \Proba(\e_i^j=-1) = 1/2$). 
The target functions are then defined by 
\begin{equation}
 \forall i\in\setn,~\forall j \in \setp,~ h_i^j = \sqrt{n}i^{-\d}\paren{\sqrt{C_1}+\e_i^j\sqrt{C_2}} \pt 
\end{equation}
Thus, all the $p$ target functions are ``close'' to a centroid function, whose coordinates on the eigenvectors of the kernel matrix are $\sqrt{n}i^{-\d}\sqrt{C_1}$, with a ``dispersion factor''  $\sqrt{C_2}$. In this setting, we can easily express the key elements for the analysis of this risk :
\begin{equation*}
 \frac{\mu_i^2}{p} = ni^{-2\d}\paren{\sqrt{C_1}+\frac{\sum_{j=1}^{p}\e_i^j}{p}\sqrt{C_2}}^2
\end{equation*}
and
\begin{equation*}
 \varsigma_i^2 = ni^{-2\d} \paren{\frac{1}{p}\sum_{j=1}^p\paren{\sqrt{C_1}+\e_i^j\sqrt{C_2}}^2 -\paren{\sqrt{C_1}+\frac{\sum_{j=1}^{p}\e_i^j}{p}\sqrt{C_2}}^2} \pt
\end{equation*}
\begin{rem}
 The theoretical analysis developed previously cannot be applied here, due to the presence of random terms, which depend on $i$, in front of the decay term $ni^{-2\d}$.
\end{rem}

\subsection{Setting B: random drawing of the input points and functions}

Assumptions~\eqref{hyp.K} and \eqref{hyp.AV} model the behaviour of the spectral elements of $f$ and $K$ as if they exactly follow the spectral elements of the kernel operator and the input function. 
Although convenient for the analysis, this setting is unlikely to hold in practice and we propose here to draw the input points $(X_i)_{i=1}^n$ and compute the risk using the eigenvalues of the kernel matrix.

We suppose here that $(X_i)_{i=1}^n$ is a sequence of i.i.d. random variables uniformly drawn on $\croch{-\pi,\pi}$. 
As in the latter section, we also suppose that  we have an i.i.d. sequence of random variables $(\e_i^j)_{i\in\setn,j\in\setp}$, independent of $(X_i)_{i\in\setn}$, following a Rademacher distribution. The target functions are then defined by 
\begin{equation}
 \forall i\in\setn,~\forall j \in \setp,~ f^j(X_i) =\paren{\sqrt{C_1}+\e_i^j\sqrt{C_2}}\absj{X_i} \pt 
\end{equation}
As stated in \citet{Wah:1990} and in \citet{gu2002smoothing}, a natural kernel to use is, with $m \in \NN^{\star}$, 
\begin{equation*}
 R(x,y) = 2\sum_{i=1}^{+\infty}\frac{\cos\paren{i(x-y)}}{i^{2m}} \pt
\end{equation*}
In this setting, the coefficients of the decomposition of $f : x \mapsto \absj{x}$ on the Fourier basis are known to be asymptotically equivalent to $i^{-2}$. Thus, this setting is a natural extension of Assumptions~\eqref{hyp.K} and \eqref{hyp.AV}, with $\b = m$---since the eigenvalues of the kernel $R$ are $i^{-2m}$---and $\d = 2$.

\subsection{Setting C: further relaxation of Assumptions~\eqref{hyp.AV}  and \eqref{2Points} in one group of tasks}
We consider the same tasks than in Setting A, but also allow the regularity of the variance to vary. 
This gives the following model, supposing that  $(\e_i^j)_{i\in\setn,j\in\setp}$ is  a sequence of i.i.d. random variables, independent of $(X_i)_{i\in\setn}$, following a Rademacher distribution:
\begin{equation*}
 \forall i\in\setn,~\forall j \in \setp,~ h_i^j = \sqrt{n}\paren{\sqrt{C_1}i^{-\d_1}+\e_i^j\sqrt{C_2}i^{-\d_2}} \pt 
\end{equation*}
We allow the variance to have a varying regularity and intensity by changing $C_2$ and $\d_2$. This gives us the following quantities of interest: for every $i\in\setn$,
\begin{equation*}
 \frac{\mu_i^2}{p} = ni^{-2\d_1}\paren{\sqrt{C_1}+\frac{\sum_{j=1}^{p}\e_i^j}{p}\sqrt{C_2}i^{-(\d_2-\d_1)}}^2
\end{equation*}
and
\begin{equation*}
  \begin{split}
     \varsigma_i^2 = ni^{-2\d_1} \Bigg(\frac{1}{p}\sum_{j=1}^p\paren{\sqrt{C_1}+\e_i^j\sqrt{C_2}i^{-(\d_2-\d_1)}}^2 ~~~~~~~~~~~~~~~~~~~~~~~~~~~~~~~~~~~~~~~~\\ ~~~~~~~~~~~~~~~~~~~~~~~~~~~~~~~~~~~~~~~~-\paren{\sqrt{C_1}+\frac{\sum_{j=1}^{p}\e_i^j}{p}\sqrt{C_2}i^{-(\d_2-\d_1)}}^2 \Bigg) \pt
  \end{split}
\end{equation*}

\subsection{Setting D: relaxation of Assumptions~\eqref{1Out} and \eqref{hyp.AV}}

Assumption~\eqref{1Out} states that we have one  of $p-1$ identical tasks and one outlier. 
We now simulate a slightly more general setting by having one cluster of $p-1$ around 0 and an outlier. 
This gives the following model, supposing that  $(\e_i^j)_{i\in\setn,j\in\setp}$ is  a sequence of i.i.d. random variables, independent of $(X_i)_{i\in\setn}$, following a Rademacher distribution:
\begin{equation*}
  \forall i \in \set{1,\dots,n},~ \forall j \in \set{1,\dots,p-1},~ h^j_i = \sqrt{n}\e_i^j i^{-2} 
\end{equation*}
and 
\begin{equation*}
  \forall i \in \set{1,\dots,n},~ \ h^p_i = \sqrt{nC_2}\e_i^p i^{-\d_2}\pt 
\end{equation*}
We allow the outlier to have a varying regularity and intensity by changing $C_2$ and $\d_2$. This gives us the following quantities of interest: for every $i\in\setn$,
\begin{equation*}
 \frac{\mu_i^2}{p} = ni^{-\d_1}\paren{\frac{\sqrt{C_1}}{p}\sum_{j=1}^{p-1}\e_i^j+\frac{\e_i^p}{p}\sqrt{C_2}i^{-(\d_2-\d_1)}}^2
\end{equation*}
and
\begin{equation*}
 \varsigma_i^2 = ni^{-\d_1} \paren{\frac{p-1}{p}C_1+\frac{1}{p}C_2i^{-2(\d_2-\d_1)} -\paren{\frac{1}{p}\sum_{j=1}^{p-1}\e_i^j+\frac{\e_i^p}{p}\sqrt{C_2}i^{-(\d_2-\d_1)}}^2} \pt
\end{equation*}

\subsection{Methodology}

 In every setting, we computed the oracle risks of both the multi-task estimator and the single-task one. 
As shown before, for instance in Equation~\eqref{optim.ind}, both the multi-task risk (which has two hyper-parameters, $\l$ and $\mu$) and the single-task risk (which has $p$ hyper-parameters, $\l^1$ to $\l^p$) can be decomposed as a sum of several functions, each depending on a unique hyper-parameter. 
We used  Newton's method to optimize each of those $p+2$ functions over, respectively, $\l$, $\mu$, $\l^1,\dots,\l^p$. 
Our stopping criterion was that the derivative of the function being optimized was inferior to $10^{-5}$, in absolute value. 
We replicated each experiment $N=100$ times. 
This gives us $N$  independent realisations of  $(\mathfrak{R}^{\star}_{\MT},\mathfrak{R}^{\star}_{\ST})$, the randomness coming from the repartition of the tasks and, in Setting~B, from the drawing of the input points $(X_i)_{i=1}^n$. 

 In Settings A and B, we  first test the hypothesis $ \mathbb{H}_0 = \set{\Proba\paren{\mathfrak{R}^{\star}_{\MT} < \mathfrak{R}^{\star}_{\ST}} < 0.5 }$ against $ \mathbb{H}_1 = \set{\Proba\paren{\mathfrak{R}^{\star}_{\MT} < \mathfrak{R}^{\star}_{\ST}} \geq 0.5 }$. 
This amounts to testing whether the median of $\frac{\mathfrak{R}^{\star}_{\MT}}{\mathfrak{R}^{\star}_{\ST}}$ is larger than one. 
For every iteration $i\in\set{1,\dots,N}$, we observe $B_i = \mathbf{1}_{\mathfrak{R}^{\star}_{\MT} < \mathfrak{R}^{\star}_{\ST}}$. 
Since the random variables $(B_i)_{i\in\set{1,\dots,N}}$ follow a Bernouilli distribution of parameter $\Proba\paren{\mathfrak{R}^{\star}_{\MT} < \mathfrak{R}^{\star}_{\ST}}$, we can apply Hoeffding's inequality \citep{Massart} and see that, for every $\e >0$, $[\bar{B}_N-\e,1]$ is a confidence interval of level $1-e^{-2N\e^2}$ for $\Proba\paren{\mathfrak{R}^{\star}_{\MT} < \mathfrak{R}^{\star}_{\ST}}$.
This leads to the following p-value: 
\begin{equation*}
  \pi_1 = \left\{ \begin{array}{lr} e^{-2N\paren{\bar{B}_N-0.5}^2}  & \textrm{if }  \bar{B}_N \geq 0.5 \virg \\ 0 & \textrm{otherwise} \pt \end{array} \right.  
\end{equation*}
 
In those two settings, we also test the hypothesis $\mathbb{H}_0 = \set{\esp{\frac{\mathfrak{R}^{\star}_{\MT}}{\mathfrak{R}^{\star}_{\ST}}} > 1}$ against $\mathbb{H}_1 = \set{\esp{\frac{\mathfrak{R}^{\star}_{\MT}}{\mathfrak{R}^{\star}_{\ST}}} \leq  1}$.
Let us denote by $\empesp{\frac{\mathfrak{R}^{\star}_{\MT}}{\mathfrak{R}^{\star}_{\ST}}}$ the empirical mean of the random variables $\frac{\mathfrak{R}^{\star}_{\MT}}{\mathfrak{R}^{\star}_{\ST}}$, $\hat{\textrm{Std}}\croch{\frac{\mathfrak{R}^{\star}_{\MT}}{\mathfrak{R}^{\star}_{\ST}}}$ the resulting standard deviation and $\Phi$ the cumulative distribution function of a standard gaussian distribution. 
Then, a classical use of the central limit theorem and of Slutsky's Lemma gives that 
\begin{equation*}
\croch{ 0, \empesp{\frac{\mathfrak{R}^{\star}_{\MT}}{\mathfrak{R}^{\star}_{\ST}}}+\frac{\e}{\sqrt{n}}\hat{\textrm{Std}}\croch{\frac{\mathfrak{R}^{\star}_{\MT}}{\mathfrak{R}^{\star}_{\ST}}}}
\end{equation*}
 is an asymptotic confidence interval of level $\Phi(\e)$ for $\esp{\frac{\mathfrak{R}^{\star}_{\MT}}{\mathfrak{R}^{\star}_{\ST}}}$.
This leads to the following asymptotic p-value: 
\begin{equation*}
\pi_2  = \Phi\croch{\sqrt{n}\paren{\empesp{\frac{\mathfrak{R}^{\star}_{\MT}}{\mathfrak{R}^{\star}_{\ST}}}-1} \hat{\textrm{Std}}\croch{\frac{\mathfrak{R}^{\star}_{\MT}}{\mathfrak{R}^{\star}_{\ST}}}^{-1} } \pt
\end{equation*}
The results of those tests are shown in Table~\ref{table.A} for Setting A and in Table~\ref{table.B} for Setting B.

In Settings C and D, we use the same asymptotic framework and show error bars corresponding to the asymptotic confidence interval 
\begin{equation*}
\croch{ \empesp{\frac{\mathfrak{R}^{\star}_{\MT}}{\mathfrak{R}^{\star}_{\ST}}} - \frac{z_{0.975}}{\sqrt{n}} \hat{\textrm{Std}}\croch{\frac{\mathfrak{R}^{\star}_{\MT}}{\mathfrak{R}^{\star}_{\ST}}}, \empesp{\frac{\mathfrak{R}^{\star}_{\MT}}{\mathfrak{R}^{\star}_{\ST}}} + \frac{z_{0.975}}{\sqrt{n}} \hat{\textrm{Std}}\croch{\frac{\mathfrak{R}^{\star}_{\MT}}{\mathfrak{R}^{\star}_{\ST}}}}  
\end{equation*}
 of level $95\%$, where $z_{\a}$ denotes the quantile of order $\a$ of the standard gaussian distribution.
The results of those simulations are shown in Figure~\ref{fig.var} for Setting~C and in Figure~\ref{fig.out} for Setting~D.

We used the following values for the parameters: $n=50$, $p=5$, $\sigma^2=1$ and $C_1 = 1$. We finally settled $\d= 2$ in Settings A and B and $\d_1 = 2$ in Settings C and D.

\begin{table}[!h]
 \centering
 \begin{tabular}{|c|c|c|c|c|c|c|c|}
\hline
  $C_2$&$r = \frac{C_2}{C_1} $ & $\b$ & $\bar{B}_{100} $ & $\pi_1$ & $\empesp{\frac{\mathfrak{R}^{\star}_{\MT}}{\mathfrak{R}^{\star}_{\ST}}}$ & $\hat{\textrm{Std}}\croch{\frac{\mathfrak{R}^{\star}_{\MT}}{\mathfrak{R}^{\star}_{\ST}}}$ & $ \pi_2$ \\
\hline
  $0.01$ & $0.01$ & 2 & $1$ &  $< 10^{-15}$ & $0.434 $ & $0.0324$ &  $< 10^{-15}$ \\
\hline and
  $0.1$ & $0.1$ & 2  & $1$ &  $< 10^{-15}$ & $0.672 $ & $0.0747$ &  $< 10^{-15}$\\
\hline
  $0.5$ & $0.5$ & 2  & $0.94$ &  $< 10^{-15}$  & $0.898$ & $0.0913$ &  $< 10^{-15}$\\
\hline
  $1$ & $1$ & 2  & $ 0.51$ &  $9.80 \times 10^{-1}$ & $1.01$ & $ 0.129$ &  $0.773$\\
\hline
  $5$ & $5$ & 2  & $ 0.38$ &  $1$  & $0.998$ & $ 0.0292$ &  $0.302$\\
\hline
  $10$ & $10$ & 2  & $ 0.42$ &  $1$ & $0.996$ & $ 0.0172$ &  $9.90\times 10^{-3}$  \\
\hline
  $100$ & $100$ & 2 & $ 0.76$ &  $1.35\times 10^{-6}$ & $0.997$ & $ 5.44\times 10^{-3}$ &  $5.97\times 10^{-10}$ \\
\hline
  $0.01$ & $0.01$ & 4  &$1$ &  $< 10^{-15}$ & $0.426$ & $0.0310$ &  $< 10^{-15}$ \\
\hline
  $0.1$ & $0.1$ & 4 &  $1$ &  $< 10^{-15}$  & $0.703$ & $0.0737$ &  $< 10^{-15}$\\
\hline
  $0.5$ & $0.5$ & 4  & $0.75$ &  $3.73\times 10^{-6}$  & $0.934$ & $0.113$ &  $1.80\times 10^{-9}$ \\
\hline
  $1$ & $1$ & 4  & $ 0.31$ &  $1$ & $1.08$ & $ 0.163$ &  $1.00$  \\
\hline
  $5$ & $5$ & 4  & $ 0.38$ &  $1$ & $1.01$ & $ 0.0439$ &  $0.965$  \\
\hline
  $10$ & $10$ & 4  & $ 0.43$ &  $1$& $0.993$ & $ 0.0304$ &  $0.0113$ \\
\hline
  $100$ & $100$ & 4  & $ 0.83$ &  $3.48\times 10^{-10}$ & $0.992$ & $ 0.0103$ &  $1.22\times 10^{-14}$ \\
\hline
\end{tabular}
\caption{Comparison of the multi-task oracle risk to the single-task oracle risk in Setting A.}
\label{table.A}
\end{table}~~~~~~~~~~~~~~~~~~~~~~~~~~~~~~~~~~~~~~~~

\begin{table}[!h]
 \centering
 \begin{tabular}{|c|c|c|c|c|c|c|c|}
\hline
  $C_2$&$r = \frac{C_2}{C_1} $ & $m$ & $\bar{B}_{100} $ &  $\pi_1$  & $\empesp{\frac{\mathfrak{R}^{\star}_{\MT}}{\mathfrak{R}^{\star}_{\ST}}}$ & $\hat{\textrm{Std}}\croch{\frac{\mathfrak{R}^{\star}_{\MT}}{\mathfrak{R}^{\star}_{\ST}}}$ &  $\pi_2$ \\
\hline
  $0.01$ & $0.01$ & 2 & $1$ &  $< 10^{-15}$ & $0.570 $ & $0.0409$ &  $< 10^{-15}$\\
\hline
  $0.1$ & $0.1$ & 2  & $1$ &  $< 10^{-15}$ & $0.745$ & $0.0333$ &  $< 10^{-15}$ \\
\hline
  $0.5$ & $0.5$ & 2  & $0.99$ &  $< 10^{-15}$ & $0.907$ & $0.0406$ &  $ < 10^{-15}$\\
\hline
  $1$ & $1$ & 2  & $ 0.80$ &  $1.52 \times 10^{-8}$  & $0.961$ & $ 0.0459 $ &  $< 10^{-15}$\\
\hline
  $5$ & $5$ & 2  & $ 0.55$ &  $0.607 $ & $0.995$ & $0.205  $ &  $2.59\times 10^{-3}$\\
\hline
  $10$ & $10$ & 2  & $ 0.53$ &  $0.835 $ & $0.996$ & $ 0.114 $ &  $6.23\times 10^{-4}$ \\
\hline
  $100$ & $100$ & 2 & $ 0.81$ &  $4.50\times 10^{-9}$ & $0.996$ & $ 6.35\times 10^{-3}$ &  $1.03\times 10^{-11}$\\
\hline
  $0.01$ & $0.01$ & 4  &$1$ &  $< 10^{-15}$& $0.527$ & $0.0409$ &  $< 10^{-15}$  \\
\hline
  $0.1$ & $0.1$ & 4 &  $1$ &  $< 10^{-15}$  & $0.756$ & $0.0534$ &  $< 10^{-15}$\\
\hline
  $0.5$ & $0.5$ & 4  & $0.93$ & $< 10^{-15}$   & $0.917$ & $0.0650$ &  $ < 10^{-15}$\\
\hline
  $1$ & $1$ & 4  & $ 0.49$ &  $1$& $1.01$ & $0.0896$ &  $ 0.855 $ \\
\hline
  $5$ & $5$ & 4  & $ 0.40$ &  $1$ & $0.997$ & $0.0295$ &  $ 0.170$\\
\hline
  $10$ & $10$ & 4  & $ 0.41$ &  $1$  & $0.998$ & $0.0179$ &  $ 0.114 $ \\
\hline
  $100$ & $100$ & 4  & $ 0.84$ & $9.10\times 10^{-11}$ & $0.994$ & $ 8.71\times 10^{-3}$ &  $7.36\times 10^{-14}$  \\
\hline
\end{tabular}
\caption{Comparison of the multi-task oracle risk to the single-task oracle risk in Setting B.}
\label{table.B}
\end{table}
~~~~~~~~~~~~~~~~~~~~~~~~~~~~~~~~~~~~~~~~

\begin{figure}
\hbox{\hspace{-1cm}\includegraphics[width=1.15\textwidth]{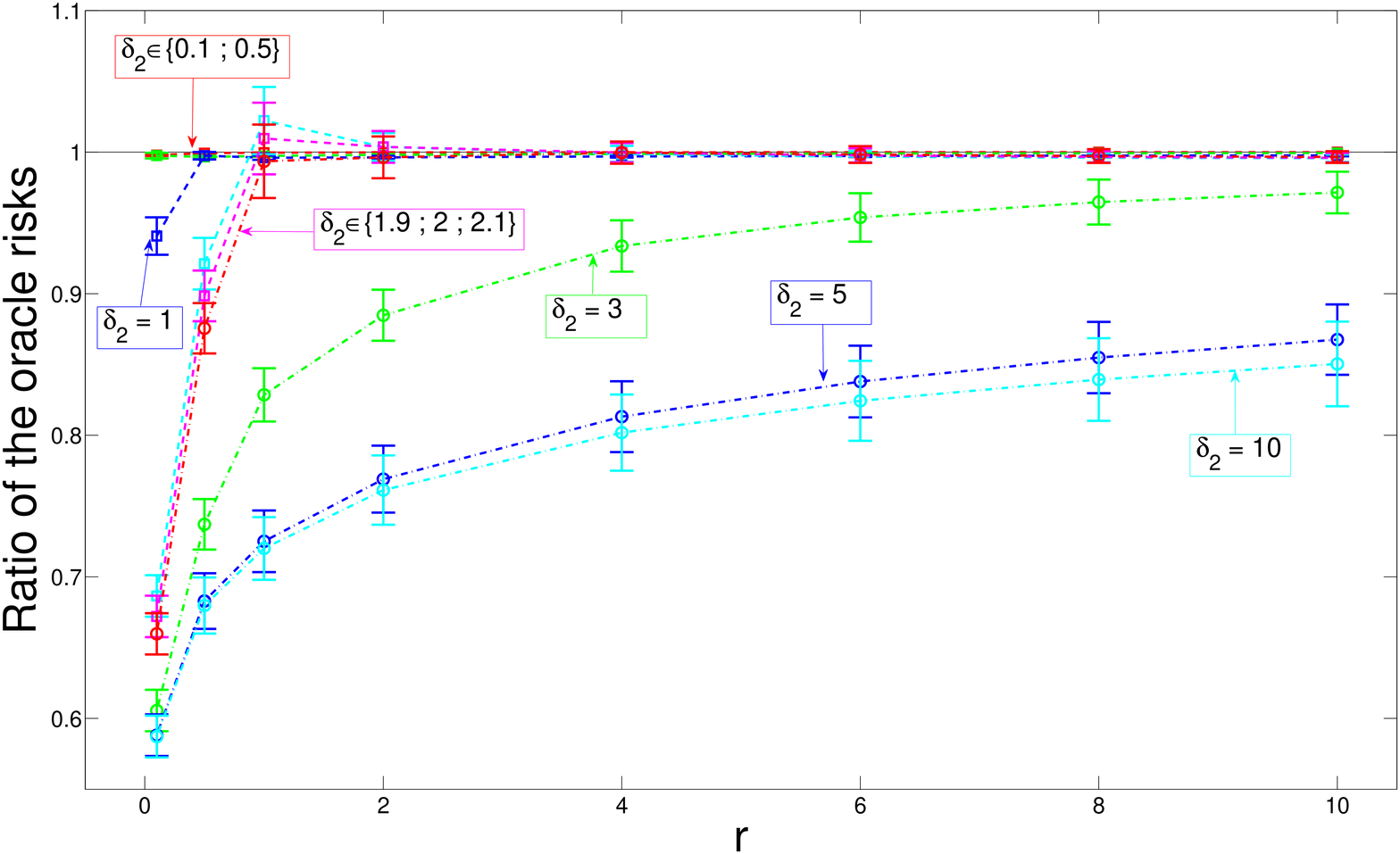}}
\caption{Further relaxation of Assumption~\eqref{2Points} (Experiment~C), improvement of multi-task compared to single-task: $\esp{\frac{\mathfrak{R}^{\star}_{\MT}}{\mathfrak{R}^{\star}_{\ST}}}$. Best seen in colour. 
}
\label{fig.var}
\end{figure}

\begin{figure}
\hbox{\hspace{-1cm}\includegraphics[width=1.15\textwidth]{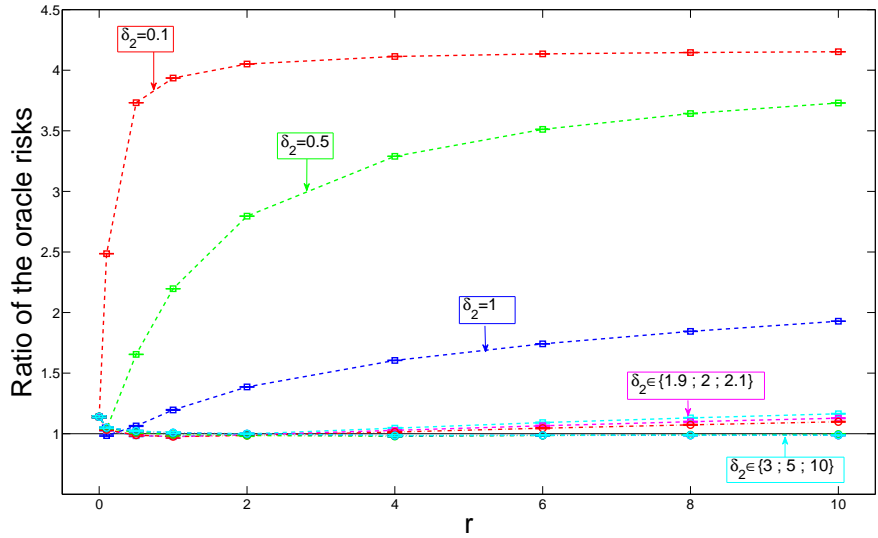}}
\caption{Relaxation of Assumption~\eqref{1Out} (Experiment~D), improvement of multi-task compared to single-task: $\esp{\frac{\mathfrak{R}^{\star}_{\MT}}{\mathfrak{R}^{\star}_{\ST}}}$. Best seen in colour.
}
\label{fig.out}
\end{figure}

\subsection{Interpretation}
  When all the tasks are grouped in one cluster (Settings A, B and C), the same phenomenon as under Assumption~\eqref{2Points} appears. 
  In situations where the mean component of the signal has more weight than the variance component (in Settings A and B, that is when $r$ is small, in Setting C, this occurs when $\d_2$ is large and $C_2$ is small) then the multi-task oracle seems to outperform the single-task one. 
  On the contrary, when the mean component of the signal is negligible compared to the variance component (likewise, this occurs in Settings A and B when $r$ is large and in Setting C when $\d_2$ is small or when $C_2$ large), then both oracles seem to perform similarly.

  Adversary settings to the multi-task oracle appear when one task is added outside of a cluster (Setting D). 
  When this outlier is less regular than the tasks belonging to the cluster (that is, when $\d_2$ is large), the single-task oracle performs better than  the multi-task one, which confirms the theoretical analysis performed in Section~\ref{sub.sec.an1Out}. 

\section{Conclusion}
This paper shows the existence of situations where the multi-task kernel ridge regression, with a perfect parameter calibration, can perform better than the single-task one. 
This happens when the tasks are distributed given simple specifications, which are studied both theoretically and on simulated examples. 

\medskip

The analysis performed here allows us to have a precise estimation of the risk of the multi-task oracle (Theorem~\ref{thm.MT.oracle}), this result holding under a few hypotheses on the regularity of the kernel, of the mean of the tasks and of its resulting variance. 
Several simple single-task settings are then investigated, with the constraint that they respect the latter assumptions. 
This theoretical grounding, backed-up by our simulated examples, allows us to understand better when and where the multi-task procedure outperforms the single-task one. 
\begin{itemize}
  \item The situation where all the regression functions are close in the RKHS (that is, their differences are extremely regular) is favorable to the multi-task procedure, when using the matrices $\M = \set{\MAV(\l,\mu),(\l,\mu)\in\R_+^2}$. 
	In this setting, the multi-task procedure can do much better than the single-task one (as if it had $p$ times more input points).
        It is also shown to never do worse (up to a multiplicative constant)~!
  \item On the contrary, when one outlier lies far apart from this cluster, this multi-task procedure suddenly performs badly, that is, arbitrarily worse than the single-task one.  
	This comes as no surprise, since the addition of a far less regular task naturally destroys the joint learning of a group of tasks. 
	In this case, the use of a multi-task procedure which clusters the tasks together (because of the choice of $\M$) is inadapted to the situation. 
\end{itemize}

\medskip

Our analysis can easily be adapted to a slightly wider set of assumptions on the tasks than the one presented here (all the tasks are grouped together, in one cluster). 
It is for instance possible to treat the case where the tasks are grouped in two (or more) clusters---when the allocation of each task to its cluster is known to the statistician, at the price of introducing more hyperparameters. 
We are still limited, though, to certain cases of hypotheses,  reflected on the set of matricial hyperparameters $\M$. 
The failure of the multi-task oracle on the case where one outlier stays outside of one group of tasks can be seen, not as the impossibility to use multi-task techniques in this situation, but rather as the fact the set of matrices used here, $\M = \set{\MAV(\l,\mu),(\l,\mu)\in\R_+^2}$, is inadapted to the situation. 
We can at least see two different solutions to this kind of inadaptation. 
First, the use of prior knowledge can help the statistician to craft an \emph{ad hoc} set $\M$.  
Second, we could seek to automatically adapt to the situation in order to learn a good set $\M$ from data. 

Learning more complex sets $\M$ is an important---but complex---challenge, that we want to address in the future. 
This question can at least be split into three (not necessarily independent) problems, that call for the elaboration of new tools: 
\begin{itemize}
  \item a careful study of the risk, to find a set $\M^{\star} \subset \mathcal{S}_p^{++} (\R)$ of candidate matrices;
  \item optimization tools, to derive an algorithm able to select a matrix in this set $\M^{\star}$;
  \item new concentration of measure results, to be able to show oracle inequalities that control the risk of the output of the algorithm.  
\end{itemize}

 \medskip
Our estimation of the multi-task oracle risk is also shown to be precise enough so that we can plug it in an oracle inequality, hereby showing the existence of a multi-task estimator that has a lower risk than the single-task oracle (under the same favorable circumstances as before).   

\medskip

Finally, it would be intereting to extend the analysis developped here to the random-design setting. 
This could be done, for instance, by using the tools brought by \citet{hsu2011analysis}, that link random-design convergence rates to fixed-design ones.

\paragraph{Acknowledgments:} The author thanks Sylvain Arlot and Francis Bach for inspireful discussions and their precious comments, which greatly helped to inprove the quality of this paper. 

\bibliographystyle{plainnat}
\bibpunct{(}{)}{;}{n}{,}{,}
\bibliography{biblio_en.bib}

\appendix
\part*{Appendices}

\section{Decomposition of the matrices $\MSD(\a,\b)$ and $\MAV(\l,\mu)$}
We now give a few technical results that were used in the former sections.
\begin{lemma} \label{lemma.M}
 The penalty used in Equation~\eqref{eq.MSD} can be obtained by using in Equation~\eqref{min.multi.M} the matrix $\MSD(\a,\b)$, such that 
\begin{equation} \label{def_MSD}
\MSD(\a,\b) =  \frac{\a}{p} \frac{\boldsymbol{1}\boldsymbol{1}\trsp}{p} + \frac{\a+p\b}{p}\paren{I_p-\frac{\boldsymbol{1}\boldsymbol{1}\trsp}{p}} \pt
\end{equation}
The penalty used in Equation~\eqref{eq.MAV} can be obtained by using in Equation~\eqref{min.multi.M} the matrix $\MAV(\a,\b)$, such that 
\begin{equation}\label{def_MAV}
 \MAV(\l,\mu) = \frac{\l}{p}\frac{\boldsymbol{1}\boldsymbol{1}\trsp}{p} + \frac{\mu}{p}\paren{I_p-\frac{\boldsymbol{1}\boldsymbol{1}\trsp}{p}}\pt
\end{equation}
\end{lemma}

\begin{proof}
 For the first part, since
 \begin{align*}
  \sum_{j=1}^p\sum_{k=1}^p \norm{ g^j - g^k}_{\F}^2 &= \sum_{j,k} \scal{g^j}{g^j}_{\F} - 2\scal{g^j}{g^k}_{\F} +  \scal{g^k}{g^k}_{\F}  \\
  &= 2p \sum_{j=1}^p\scal{g^j}{g^j}_{\F} - 2\sum_{j,k}\scal{g^j}{g^k}_{\F}\virg	
 \end{align*}
the penalty term of Equation~\eqref{eq.MSD} can be written as
\begin{equation*}
 \frac{\a}{p} \sum_{j=1}^p \scal{g^j}{g^j}_{\F} + \b\sum_{j=1}^p \scal{g^j}{g^j}_{\F} - \frac{\b}{p} \sum_{j,k}\scal{g^j}{g^k}_{\F}\virg
\end{equation*}
leading to the matrix 
\begin{equation*}
\frac{\a+p\b}{p} I_p -  \frac{\b}{p} \boldsymbol{1}\boldsymbol{1}\trsp =  \frac{\a}{p} \frac{\boldsymbol{1}\boldsymbol{1}\trsp}{p} + \frac{\a+p\b}{p}\paren{I_p-\frac{\boldsymbol{1}\boldsymbol{1}\trsp}{p}} = \MSD(\a,\b) \pt
\end{equation*}

 For the second part, since 
\begin{equation*}
 \norm{\sum_{j=1}^p  g^j}_{\F}^2 = \sum_{j,k} \scal{g^j}{g^k}_{\F} \virg	
\end{equation*}
the penalty term of Equation~\eqref{eq.MAV} can be written as
\begin{equation*}
\frac{\l}{p^2}\sum_{j,k} \scal{g^j}{g^k}_{\F} + \frac{\mu}{p}\sum_{j=1}^p \scal{g^j}{g^j}_{\F} - \frac{\mu}{p^2} \sum_{j,k} \scal{g^j}{g^k}_{\F} \virg
\end{equation*}
leading to the matrix 
\begin{equation*}
\frac{\l-\mu}{p^2}\boldsymbol{1}\boldsymbol{1}\trsp + \frac{\mu}{p}I_p = \frac{\l}{p}\frac{\boldsymbol{1}\boldsymbol{1}\trsp}{p} + \frac{\mu}{p}\paren{I_p-\frac{\boldsymbol{1}\boldsymbol{1}\trsp}{p}}  = \MAV(\l,\mu) \pt
\end{equation*}

\end{proof}

\section{Useful control of some sums}

Let us introduce the following integrals :
\begin{align*}
I_1 = I_1(\b,\d)&=  \int_0^{+\infty} \frac{u^{\frac{1-2\d}{2\b}+1}}{(1+u)^2}du \virg \\
I_2 = I_2(\b) &=  \int_{0}^{+\infty} \frac{u^{\frac{1}{2\b}-1}}{(1+u)^2}du = I_1(\b,0) \pt
\end{align*}
Under Assumption~\eqref{hyp.minimax}, both integrals converge. We also introduce their discrete counterparts. For every $n\in\NN^{\star}$ and every $\l\in\R_+$ : 
  \begin{align*}
   S_1(n,\l) &=  \sum_{i=1}^n\frac{i^{4\b-2\d}}{(1+\l i^{2\b})^2} \virg \\
   S_2(n,\l) &= \sum_{i=1}^n \frac{1}{\left(1+\l i^{2\b}\right)^2} \pt
  \end{align*}

We here give a first elementary technical result.

\begin{lemma} \label{lemma.tstar}
 The map defined on $\R_+$ by
\begin{equation*}
 t \mapsto \frac{t^{4\b-2\d}}{(1+\l t^{2\b})^2}
\end{equation*}
is positive, increasing on $[0,t^{\star}]$ and decreasing on $[t^{\star},+\infty)$ to 0, with
\begin{equation*}
 t^{\star} = \paren{\frac{4\b-2\d}{2\d\l}}^{1/2\b}
\end{equation*}
\end{lemma}
\begin{proof}
This map is nonnegative and converges to $0$ in 0 and $+\infty$. Furthermore
\begin{align*}
 \frac{d}{dt}\paren{\frac{t^{4\b-2\d}}{(1+\l t^{2\b})^2}} &= (4\b-2\d)\frac{t^{4\b-2\d-1}}{(1+\l t^{2\b})^2} - 4\b\l t^{2\b-1}\frac{t^{4\b-2\d}}{(1+\l t^{2\b})^3}\\
&= \frac{t^{4\b-2\d-1}}{(1+\l t^{2\b})^3} \croch{(4\b-2\d)(1+\l t^{2\b}) - 4\b\l t^{2\b}} \\
&= \frac{t^{4\b-2\d-1}}{(1+\l t^{2\b})^3} \croch{4\b+4\b\l t^{2\b}-2\d-2\d\l t^{2\b}- 4\b\l t^{2\b}} \\
&= \frac{t^{4\b-2\d-1}}{(1+\l t^{2\b})^3} \croch{(4\b-2\d)-2\d\l t^{2\b}} \pt
\end{align*}
The only parameter $t^{\star}$ that cancels out this equation is
\begin{equation*}
 t^{\star} = \paren{\frac{4\b-2\d}{2\d\l}}^{1/2\b} \pt
\end{equation*}
\end{proof}

We now give a serie of technical results to control $I_1$, $I_2$, $S_1$ and $S_2$, which will be useful in the following sections.

\begin{lemma}\label{lemma.int1}
 \begin{equation*}
  \int_0^{+\infty} \frac{t^{4\b-2\d}}{(1+\l t^{2\b})^2} dt = \frac{\l^{(2\d-1)/2\b}}{2\b\l^2}\int_0^{+\infty} \frac{u^{\frac{1-2\d}{2\b}+1}}{(1+u)^2}du = \frac{\l^{(2\d-1)/2\b}}{2\b\l^2} I_1 \pt
 \end{equation*}
\end{lemma}
\begin{proof}
Apply the change of variables $u = \l t^{2\b}$ see  \citet{bach2012_column_sampling} for more details.
\end{proof}

\begin{lemma}\label{lemma.int2}
  \begin{equation*}
  \int_0^{+\infty} \frac{1}{(1+\l t^{2\b})^2} dt = \frac{\l^{-1/2\b}}{2\b} \int_{0}^{+\infty} \frac{u^{\frac{1-2\b}{2\b}}}{(1+u)^2}du = \frac{\l^{-1/2\b}}{2\b} I_2 \pt
 \end{equation*}
\end{lemma}
\begin{proof}
Apply the change of variables $u = \l t^{2\b}$ see  \citet{bach2012_column_sampling} for more details.
\end{proof}

\begin{lemma}\label{lemma.S2}
 We have the following bounds $S_2$. For every $n \in \NN^{\star}$ and every $\l \in \R_+^{\star}$,
\begin{itemize}
\item 
 \begin{equation*}
 S_2(n,\l) \leq  \frac{\l^{-1/2\b}}{2\b} I_2  \pt
 \end{equation*}
\item 
 \begin{equation*}
 S_2(n,\l) \geq  \int_1^{n+1}  \frac{1}{(1+\l t^{2\b})^2} dt  \pt
 \end{equation*}
\end{itemize}
\end{lemma}
\begin{proof}
To show the first point we just remark that 
\begin{equation*}
 S_2(n,\l) = \sum_{i=1}^n \frac{1}{\left(1+\l i^{2\b}\right)^2} \leq \int_0^{n} \frac{1}{(1+\l t^{2\b})^2} dt \leq \int_0^{+\infty} \frac{1}{(1+\l t^{2\b})^2} dt \pt
\end{equation*}
The second point is likewise straightforward.
\end{proof}

\begin{lemma} \label{lemma.S1}
 We have the following bounds on $S_1$  : for every $n \in \NN^{\star}$, every $(\b,\d)\in\R_+^2$ such that $4\b>2\d$ and every $\l \in \R_+^{\star}$,
 \begin{equation*}
 S_1(n,\l) \leq   \frac{\l^{(2\d-1)/2\b}}{\b\l^2} I_1 \virg
 \end{equation*}
 Furthermore, let 
 \begin{equation*}
 t^{\star} = \paren{\frac{4\b-2\d}{2\d\l}}^{1/2\b}
\end{equation*}
and $n^{\star} = \lfloor  t^{\star} \rfloor$.  
\begin{itemize}
 \item If $n^{\star} < n-1$
 \begin{equation*}
 S_1(n,\l) \geq   \int_0^{n+1} \frac{t^{4\b-2\d}}{(1+\l t^{2\b})^2} dt   - \int_{n^{\star}}^{n^{\star}+2} \frac{t^{4\b-2\d}}{(1+\l t^{2\b})^2} dt\ptvirg
 \end{equation*}
 \item while if $n^{\star} \geq n$
  \begin{equation*}
   S_1(n,\l) \geq   \int_0^{n} \frac{t^{4\b-2\d}}{(1+\l t^{2\b})^2} dt  \pt
  \end{equation*}
\end{itemize}
\end{lemma}
\begin{proof}
 Lemma \ref{lemma.tstar} shows that $t \mapsto \frac{t^{4\b-2\d}}{(1+\l t^{2\b})^2}$ is increasing on $[0,t^{\star}]$ and decreasing on $[t^{\star},+\infty($. Thus we have the following comparisons :
\begin{equation*}
     \int_0^{n^{\star}} \frac{t^{4\b-2\d}}{(1+\l t^{2\b})^2} dt \leq \sum_{i=1}^{n^{\star}}\frac{i^{4\b-2\d}}{(1+\l i^{2\b})^2}  \leq  \int_1^{n^{\star}+1} \frac{t^{4\b-2\d}}{(1+\l t^{2\b})^2} dt
\end{equation*}
and
\begin{equation*}
     \int_{n^{\star}+2}^{n+1} \frac{t^{4\b-2\d}}{(1+\l t^{2\b})^2} dt \leq \sum_{i=n^{\star}+1}^n\frac{i^{4\b-2\d}}{(1+\l i^{2\b})^2}  \leq  \int_{n^{\star}}^n \frac{t^{4\b-2\d}}{(1+\l t^{2\b})^2} dt \pt
\end{equation*}
By adding those two lines we get 
 \begin{align*}
 S_1(n,\l) &=  \sum_{i=1}^n\frac{i^{4\b-2\d}}{(1+\l i^{2\b})^2}  \leq \int_1^{n^{\star}+1} \frac{t^{4\b-2\d}}{(1+\l t^{2\b})^2}dt + \int_{n^{\star}}^n \frac{t^{4\b-2\d}}{(1+\l t^{2\b})^2} dt \\
           &\leq 2 \int_0^{+\infty} \frac{t^{4\b-2\d}}{(1+\l t^{2\b})^2} dt \virg
 \end{align*}
which shows the first point. We also get, if $n^{\star} < n-1$
 \begin{align*}
 S_1(n,\l) & \geq   \int_0^{n^{\star}} \frac{t^{4\b-2\d}}{(1+\l t^{2\b})^2} dt +  \int_{n^{\star}+2}^{n+1} \frac{t^{4\b-2\d}}{(1+\l t^{2\b})^2} dt \\
    & \geq   \int_0^{n+1} \frac{t^{4\b-2\d}}{(1+\l t^{2\b})^2} dt   - \int_{n^{\star}}^{n^{\star}+2} \frac{t^{4\b-2\d}}{(1+\l t^{2\b})^2} dt\pt
 \end{align*}
The last point is evident, since if $n^{\star} \geq n$ the integrand is increasing on $[0,n]$.
\end{proof}

\section{Proof of Property~\ref{prop.maj.risk} \label{app.proof.prop.maj.risk}}
Let $n$ and $p$ be integers, $\s$, $\b$ and $\d$ real numbers such that \eqref{hyp.minimax} hold. We want to study the value and the location of the infimum on $\R_+$ of 
\begin{equation} \label{eq.def.R}
 \l \longmapsto R(n,p,\s^2,\l,\b,\d,C) = C\l^2 \sum_{i=1}^n\frac{i^{4\b-2\d}}{(1+\l i^{2\b})^2}+ \frac{\s^2}{np} \sum_{i=1}^n \frac{1}{\left(1+\l i^{2\b}\right)^2}
\end{equation}

\begin{property}
 For every $\l$ in $\R_+$, we have 
\begin{equation} \label{eq.maj.R}
 R(n,p,\s^2,\l,\b,\d,C) \leq \frac{CI_1}{\b }\l^{(2\d-1)/2\b} + \frac{\s^2I_2}{2\b np}\l^{-1/2\b} \pt
\end{equation}
\end{property}
\begin{proof}
 This is a straightforward application of the majorations of the finite sums by integrals given in Lemmas~\ref{lemma.S2} and \ref{lemma.S1}, together with the change of variables done in Lemmas~\ref{lemma.int1} and \ref{lemma.int2}.
\end{proof}

\begin{lemma}\label{lemma.optimize}
Let $A \in \R_+$, the minimum over $\R_+^{\star}$ of $\l \mapsto \l^{(2\d-1)/2\b}+A\l^{-1/2\b}$ is attained for 
\begin{equation*}
 \l^{\star} = \paren{\frac{A}{2\d-1}}^{\b/\d}
\end{equation*}
and has for value
\begin{equation*}
A^{1-(1/2\d)} \frac{2\d}{(2\d-1)^{1-(1/2\d)}} \pt
\end{equation*}
\end{lemma}
\begin{proof}
 This mapping is differentiable and has $+\infty$ for limit in $0$ and in $+\infty$. Then 
\begin{equation*}
 \frac{d}{d\l}\paren{\l^{2\d/(2\d-1)}+A\l^{-1/2\b}} = \frac{1}{\l}\paren{\frac{2\d-1}{2\b}\l^{(2\d-1)/2\b}-\frac{A}{2\b}\l^{-1/2\b}} \pt
\end{equation*}
We see there is only one minimizer $\l^{\star}$ verifying
\begin{equation*}
\begin{array}{lrcl}%
 & \frac{2\d-1}{2\b}(\l^{\star})^{(2\d-1)/2\b} &=& \frac{A}{2\b}(\l^{\star})^{-1/2\b} \\%
\Leftrightarrow &(2\d-1)^{2\b}(\l^{\star})^{2\d-1} &=& A^{2\b}(\l^{\star})^{-1} \\%
\Leftrightarrow &(\l^{\star})^{2\d} &=& \frac{A^{2\b}}{(2\d-1)^{2\b}} \\%
\Leftrightarrow &\l^{\star} &=& \paren{\frac{A}{2\d-1}}^{\b/\d} \pt
\end{array}%
\end{equation*}
Pluging-in the value of $\l^{\star}$ leads to the optimal value
\begin{align*}
 \paren{\frac{A}{2\d-1}}^{(2\d-1)/2\d} + A\paren{\frac{A}{2\d-1}}^{-1	/2\d} &= A^{(2\d-1)/2\d} \paren{(2\d-1)^{(1/2\d)-1}+(2\d-1)^{1/2\d}}\\
&= A^{(2\d-1)/2\d} (2\d-1)^{1/2\d} \paren{\frac{1}{2\d-1}+1}\\
&= A^{(2\d-1)/2\d} (2\d-1)^{1/2\d} \paren{\frac{2\d}{2\d-1}}\\
&= A^{(2\d-1)/2\d} \frac{2\d}{(2\d-1)^{1-(1/2\d)}} \pt
\end{align*}
\end{proof}

\begin{definition}
 To simplify notations, since this quantity depends only on $\b$ and $\d$ and appears throughout the paper, we will use the following notation : 
 \begin{equation}\label{def.kbd}
  \kappa(\b,\d) = I_1(\b,\d)^{1/2\d}I_2(\b)^{1-(1/2\d)}(2\d-1)^{1/2\d}\frac{\d}{\b(2\d-1)} \pt
 \end{equation}
\end{definition}

We now prove Property~\ref{prop.maj.risk}

\begin{proof}
 First $R(n,p,\s^2,0,\b,\d,C) = \frac{\s^2}{p}$, so that $R^{\star}(n,p,\s^2,\b,\d,C) \leq  \frac{\s^2}{p}$. Then, the right-hand side of Equation~\eqref{eq.maj.R} can be written as 
 \begin{equation*}
  \frac{CI_1}{\b}\croch{\l^{(2\d-1)/2\b} + \frac{\s^2I_2}{2npCI_1}\l^{-1/2\b}}\pt
 \end{equation*}
Consequently, Lemma \ref{lemma.optimize} implies that the optimal value of this upper bound with respect to $\l$ is 
\begin{equation*}
  \frac{CI_1}{\b} \paren{\frac{\s^2I_2}{2npCI_1}}^{1-(1/2\d)}\frac{2\d}{(2\d-1)^{1-(1/2\d)}} \virg
\end{equation*}
which is exactly the right-hand side of Equation~\eqref{eq.maj.R.opt}.
\end{proof}

\section{Proof of Property~\ref{prop.param.reg.maj} \label{proof.prop.param.reg.maj}}

In order to perform this analysis we observe that $R$ is composed of two factors :
\begin{itemize}
 \item a bias factor \matheq{C\l^2 \sum_{i=1}^n\frac{i^{4\b-2\d}}{(1+\l i^{2\b})^2}}, which is an increasing function of $\l$;
 \item a variance factor \matheq{\frac{\s^2}{np} \sum_{i=1}^n \frac{1}{\left(1+\l i^{2\b}\right)^2 }}, which is a convex, decreasing function of $\l$. 
\end{itemize}
We show that, if $\l$ is too large, then the bias term exceeds the upper bound on  $R^{\star}(n,p,\s^2,\b,\d,C)$ given in Equation~\eqref{eq.maj.R.opt}.

\begin{proof}
 We see that, using Equation~\eqref{eq.def.R}, for every $\l \in \R_+$, 
 \begin{equation*}
   R(n,p,\s^2,\l,\b,\d,C) \geq C\frac{\l^2}{(1+\l)^2} \pt
 \end{equation*}
The right-hand side of this equation is increasing. Thus, if a real number $\e$ matches this bound with the upper bound of $R^{\star}$, that is, 
\begin{equation*}
 C\frac{\e^2}{(1+\e)^2} =  \frac{1}{np}\times (np)^{1/2\d}C^{(1/2\d)}2^{1/2\d}\kbd \virg
\end{equation*}
 we can state that the infimum of $R$ is attained by a parameter $\l^{\star} \in [0,\e]$. The latter equation is equivalent to 
\begin{equation*}
 \e^2 =  A\paren{\frac{np}{\s^2}}^{(1/2\d)-1}(1+\e)^2 \virg
\end{equation*}
with
\begin{equation*}
 A = C^{(1/2\d)-1}2^{1/2\d}\kbd \pt
\end{equation*}
This leads to 
\begin{equation*}
 \e\paren{1-\sqrt{A}\paren{\frac{np}{\s^2}}^{(1/4\d)-1/2}} = \sqrt{A}\paren{\frac{np}{\s^2}}^{(1/4\d)-2}  \virg
\end{equation*}
so that  if $\sqrt{A} \paren{\frac{np}{\s^2}}^{(1/4\d)-1/2}<1$ that is, if 
 \begin{equation*}
  \frac{np}{\s^2} >  \frac{1}{C} \times 2^{\frac{1}{2\d-1}}\times\kbd^{\frac{2\d}{2\d-1}} \virg
 \end{equation*}
then
\begin{equation}\label{eq.def.epsilon}
 \e = \frac{ \sqrt{A} \paren{\frac{np}{\s^2}}^{(1/4\d)-1/2}}{1-\sqrt{A}\paren{\frac{np}{\s^2}}^{(1/4\d)-1/2}} =  \sqrt{A} \paren{\frac{np}{\s^2}}^{(1/4\d)-1/2}\paren{1+\eta\paren{\frac{np}{\s^2}}} \virg
\end{equation}
where $\eta(x)$ goes to 0 as $x$ goes to $+\infty$.
\end{proof}

\section{On the way to showing Property~\ref{prop.min.risk}}
The proof of Property~\ref{prop.min.risk} uses two results that we give here. 
\subsection{Control of the risk on $\croch{0,n^{-2\b}}$ \label{sec.proof.prop.min.risk1}}
\begin{property}\label{prop.min.risk.l.small}
 For every $n$, $p$, $\s^2$, $C$, $\d$ and $\b$, we have 
  \begin{equation*}
   \inf_{\l\in\croch{0,n^{-2\b}}} \set{R(n,p,\s^2,\l,\b,\d,C)} \geq \frac{\s^2}{4p} \pt
  \end{equation*}
\end{property}

\begin{proof}
 For every $\l\in\croch{0,n^{-2\b}}$ we have 
  \begin{align*}
    R(n,p,\s^2,\l,\b,\d,C) &\geq  \frac{\s^2}{np} \sum_{i=1}^n \frac{1}{\left(1+\l i^{2\b}\right)^2} \\
			   &\geq  \frac{\s^2}{p}\times\frac{1}{n} \sum_{i=1}^n \frac{1}{\left(1+ \paren{\frac{i}{n}}^{2\b}\right)^2} \\
                           &\geq  \frac{\s^2}{4p} \pt
  \end{align*}
\end{proof}

\subsection{Control of the risk on $\croch{n^{-2\b},\e\paren{\frac{np}{\s^2}}}$ \label{sec.proof.prop.min.risk2}}

\begin{property}\label{prop.min.risk.l.large}
 There exists an integer $N$ and a constant $\a\in(0,1)$ such that for every $(n,p,\s^2)$ such that $np/\s^2 \geq N$, every $(\b,\d)\in\R_+^2$ such that $4\b>2\d>1$ and every $\l \in [n^{-2\b},\e\paren{\frac{np}{\s^2}}]$ we have 
 \begin{equation} \label{eq.min.R}
  R(n,p,\s^2,\l,\b,\d,C) \geq \a \paren{\frac{CI_1}{\b }\l^{(2\d-1)/2\b} + \frac{\s^2I_2}{2\b np}\l^{-1/2\b}} \pt
 \end{equation}

\end{property}

\begin{proof}
 We seek to minor the two sums composing $R$, which was definded in Equation~\eqref{eq.def.R}, by their integral counterparts, uniformly on $[n^{-2\b},\e\paren{\frac{np}{\s^2}}]$. The technical details are exposed in Lemmas~\ref{lemma.S2} and \ref{lemma.S1}.

 For the first sum, using Lemma~\ref{lemma.S2}, we have that
\begin{align*}
 \sum_{i=1}^n\frac{1}{(1+\l i^{2\b})^2} &\geq  \int_0^{n+1}  \frac{1}{(1+\l t^{2\b})^2} dt - \int_0^{1}  \frac{1}{(1+\l t^{2\b})^2} dt \\
                                        &\geq \int_0^{+\infty}  \frac{1}{(1+\l t^{2\b})^2} dt - \int_{n+1}^{+\infty}  \frac{1}{(1+\l t^{2\b})^2} dt -\int_0^{1}  \frac{1}{(1+\l t^{2\b})^2} dt\pt \\
\end{align*}
 
First, with the change of variables  $u = \l t^{2\b}$ \citep[as in][]{bach2012_column_sampling},
 \begin{align*}
   \int_{n+1}^{+\infty}\frac{1}{(1+\l t^{2\b})^2} dt &= \int_0^{+\infty}  \frac{1}{(1+\l t^{2\b})^2} dt \frac{  \int_{n+1}^{+\infty}\frac{1}{(1+\l t^{2\b})^2}dt}{ \int_0^{+\infty}  \frac{1}{(1+\l t^{2\b})^2} dt } \\
                                            &= \int_0^{+\infty}  \frac{1}{(1+\l t^{2\b})^2} dt \frac{\int_{\l(n+1)^{2\b}}^{+\infty} \frac{u^{\frac{1}{2\b}-1}}{(1+u)^2}du}{\int_{0}^{+\infty} \frac{u^{\frac{1}{2\b}-1}}{(1+u)^2}du} \\
					    &\leq \int_0^{+\infty}  \frac{1}{(1+\l t^{2\b})^2} dt \frac{\int_{1}^{+\infty} \frac{u^{\frac{1}{2\b}-1}}{(1+u)^2}du}{\int_{0}^{+\infty} \frac{u^{\frac{1}{2\b}-1}}{(1+u)^2}du} \virg
 \end{align*}
 since $\l \geq n^{-2\b}$.

 We also have , with the change of variables  $u = \l t^{2\b}$ \citep[as in][]{bach2012_column_sampling},
 \begin{align*}
  \int_0^{1}  \frac{1}{(1+\l t^{2\b})^2} dt &= \int_0^{+\infty}  \frac{1}{(1+\l t^{2\b})^2} dt \frac{\int_0^{1}  \frac{1}{(1+\l t^{2\b})^2}dt}{ \int_0^{+\infty}  \frac{1}{(1+\l t^{2\b})^2} dt } \\
                                            &= \int_0^{+\infty}  \frac{1}{(1+\l t^{2\b})^2} dt \frac{\int_{0}^{\l} \frac{u^{\frac{1}{2\b}-1}}{(1+u)^2}du}{\int_{0}^{+\infty} \frac{u^{\frac{1}{2\b}-1}}{(1+u)^2}du} \\
					    &\leq \int_0^{+\infty}  \frac{1}{(1+\l t^{2\b})^2} dt \frac{\int_{0}^{\e} \frac{u^{\frac{1}{2\b}-1}}{(1+u)^2}du}{\int_{0}^{+\infty} \frac{u^{\frac{1}{2\b}-1}}{(1+u)^2}du} \pt
 \end{align*}
 Since $\e$, which was defined in Equation~\eqref{eq.def.epsilon}, verifies $\e(x) \longrightarrow 0$ as $x \longrightarrow +\infty$, we get 
 \begin{equation*}
  \frac{\int_{0}^{\e(x)} \frac{u^{\frac{1}{2\b}-1}}{(1+u)^2}du}{\int_{0}^{+\infty} \frac{u^{\frac{1}{2\b}-1}}{(1+u)^2}du} \converge_{x\to+\infty} 0 \pt
 \end{equation*}
 All those arguments imply that there exists an integer $n_1$ and real number $c_1 \in (0,1)$ such that, for every $(n,p,\s^2)$ such that $np/\s^2 \geq n_3$ and for every $\l \in [n^{-2\b},\e\paren{\frac{np}{\s^2}}]$,  
\begin{equation*}
 \sum_{i=1}^n\frac{1}{(1+\l i^{2\b})^2} \geq  c_1 \int_0^{+\infty}  \frac{1}{(1+\l t^{2\b})^2} dt\pt
\end{equation*}

For the second sum we carry a similar analysis, using Lemma~\ref{lemma.S1} instead of Lemma~\ref{lemma.S2}. First, supposing that $4\b > 2\d$, we know that 
\begin{equation*}
   \frac{\left\lfloor\paren{\frac{4\b-2\d}{2\d\l}}^{1/2\b}\right\rfloor}{\paren{\frac{4\b-2\d}{2\d\l}}^{1/2\b}} \converge_{\l\to 0} 1 \pt
 \end{equation*}
Since $\e(np/\s^2)$ goes to $0$ as $ np/\s^2$ goes to $+\infty$. Consequently, let $\zeta >0$ and $n_3$ be an integer such that for every $(n,p,\s^2)$ such that $np/\s^2 \geq n_3$,  and every $\l \in [n^{-2\b},\e\paren{\frac{np}{\s^2}}]$, we have 
\begin{equation*}
   \absj{\frac{\left\lfloor\paren{\frac{4\b-2\d}{2\d\l}}^{1/2\b}\right\rfloor}{\paren{\frac{4\b-2\d}{2\d\l}}^{1/2\b}} -1} < \zeta \et \absj{\frac{\left\lfloor\paren{\frac{4\b-2\d}{2\d\l}}^{1/2\b}\right\rfloor+2}{\paren{\frac{4\b-2\d}{2\d\l}}^{1/2\b}} -1} < \zeta \pt
 \end{equation*}
Consequently, for every $(n,p,\s^2)$ such that $np/\s^2 \geq n_3$  and every $\l \in [n^{-2\b},\e\paren{\frac{np}{\s^2}}]$, we have (with $t^{\star} = \paren{\frac{4\b-2\d}{2\d\l}}^{1/2\b}$ and $n^{\star} = \lfloor  t^{\star} \rfloor$) : 
\begin{equation*}
 n^{\star} \geq (1-\zeta) \paren{\frac{4\b-2\d}{2\d\l}}^{1/2\b} = z_1 ~\textrm{ and }~ n^{\star}+2 \leq (1+\zeta) \paren{\frac{4\b-2\d}{2\d\l}}^{1/2\b} = z_2\pt
\end{equation*}
We can remark that $\l z_1^{2\b}$ and $\l z_2^{2\b}$ do not depend on $\l$. Consequently, for every $(n,p,\s^2)$ such that $np/\s^2 \geq n_3$  and every $\l \in [n^{-2\b},\e\paren{\frac{np}{\s^2}}]$, we get 
\begin{equation*}
 \int_{n^{\star}}^{n^{\star}+2} \frac{t^{4\b-2\d}}{(1+\l t^{2\b})^2} dt \leq \int_{z_1}^{z_2} \frac{t^{4\b-2\d}}{(1+\l t^{2\b})^2} dt \pt
\end{equation*}
We finally see that 
\begin{align*}
 \int_{z_1}^{z_2} \frac{t^{4\b-2\d}}{(1+\l t^{2\b})^2} dt &=  \int_0^{+\infty}  \frac{t^{4\b-2\d}}{(1+\l t^{2\b})^2} dt \frac{\int_{z_1}^{z_2} \frac{t^{4\b-2\d}}{(1+\l t^{2\b})^2} dt}{  \int_0^{+\infty}  \frac{t^{4\b-2\d}}{(1+\l t^{2\b})^2} dt}\\
                                                          &= \int_0^{+\infty}  \frac{t^{4\b-2\d}}{(1+\l t^{2\b})^2} dt \frac{\int_{\l z_1^{2\b}}^{\l z_2^{2\b}} \frac{u^{\frac{1-2\d}{2\b}+1}}{(1+u)^2}du}{\int_0^{+\infty} \frac{u^{\frac{1-2\d}{2\b}+1}}{(1+u)^2}du}\\
                                                          &= c_3 \int_0^{+\infty}  \frac{t^{4\b-2\d}}{(1+\l t^{2\b})^2} dt \virg
\end{align*}
with 
\begin{equation*}
 c_3 =  \frac{\int_{\l z_1^{2\b}}^{\l z_2^{2\b}} \frac{u^{\frac{1-2\d}{2\b}+1}}{(1+u)^2}du}{\int_0^{+\infty} \frac{u^{\frac{1-2\d}{2\b}+1}}{(1+u)^2}du} \in (0,1)
\end{equation*}
being independent of $\l$ and arbitrarily close to 0. Thus, we have that, using Lemma~\ref{lemma.S1}, 
\begin{itemize}
 \item if $n^{\star} \geq n-1$: 
  \begin{align*}
   \sum_{i=1}^n\frac{i^{4\b-2\d}}{(1+\l i^{2\b})^2} &\geq  \int_0^{n} \frac{t^{4\b-2\d}}{(1+\l t^{2\b})^2} dt \\
                                                    &\geq  \int_0^{+\infty} \frac{t^{4\b-2\d}}{(1+\l t^{2\b})^2} dt - \int_n^{+\infty} \frac{t^{4\b-2\d}}{(1+\l t^{2\b})^2} dt \ptvirg
  \end{align*}
 \item if $n^{\star} < n-1$ and $np/\s^2 \geq n_3$:
  \begin{align*}
   &\sum_{i=1}^n\frac{i^{4\b-2\d}}{(1+\l i^{2\b})^2} \\
    \geq~~&  \int_0^{n} \frac{t^{4\b-2\d}}{(1+\l t^{2\b})^2} dt -c_3 \int_0^{+\infty} \frac{t^{4\b-2\d}}{(1+\l t^{2\b})^2} dt \\
                                                    \geq~~&  \int_0^{+\infty} \frac{t^{4\b-2\d}}{(1+\l t^{2\b})^2} dt - \int_n^{+\infty} \frac{t^{4\b-2\d}}{(1+\l t^{2\b})^2} dt -c_3 \int_0^{+\infty} \frac{t^{4\b-2\d}}{(1+\l t^{2\b})^2} dt \pt
  \end{align*}
\end{itemize}
With the change of variables  $u = \l t^{2\b}$ \citep[as in][]{bach2012_column_sampling},
 \begin{align*}
   \int_{n}^{+\infty}\frac{1}{(1+\l t^{2\b})^2} dt &= \int_0^{+\infty}  \frac{t^{4\b-2\d}}{(1+\l t^{2\b})^2} dt \frac{  \int_{n}^{+\infty}\frac{t^{4\b-2\d}}{(1+\l t^{2\b})^2}dt}{ \int_0^{+\infty}  \frac{t^{4\b-2\d}}{(1+\l t^{2\b})^2} dt } \\
                                            &= \int_0^{+\infty}  \frac{t^{4\b-2\d}}{(1+\l t^{2\b})^2} dt \frac{\int_{\l n^{2\b}}^{+\infty} \frac{u^{\frac{1-2\d}{2\b}+1}}{(1+u)^2}du}{\int_{0}^{+\infty} \frac{u^{\frac{1-2\d}{2\b}+1}}{(1+u)^2}du} \\
					    &\leq \int_0^{+\infty}  \frac{t^{4\b-2\d}}{(1+\l t^{2\b})^2} dt \frac{\int_{1}^{+\infty} \frac{u^{\frac{1-2\d}{2\b}+1}}{(1+u)^2}du}{\int_{0}^{+\infty} \frac{u^{\frac{1-2\d}{2\b}+1}}{(1+u)^2}du} \virg
 \end{align*}
 since $\l \geq n^{-2\b}$.
This implies that there exists an integer $n_2$ and real number $c_2 \in (0,1)$ such that, for every $(n,p,\s^2)$ such that $np/\s^2 \geq n_2$ and for every $\l \in [n^{-2\b},\e\paren{\frac{np}{\s^2}}]$,  
\begin{equation*}
 \sum_{i=1}^n\frac{i^{4\b-2\d}}{(1+\l i^{2\b})^2} \geq  c_2  \int_0^{+\infty}  \frac{t^{4\b-2\d}}{(1+\l t^{2\b})^2} dt\pt
\end{equation*}

By taking $N = \max(n_1,n_2)$ and $\a = \min(c_1,c_2)$, we have that for every $(n,p,\s^2)$ such that $np/\s^2 \geq N$ and every $\l \in [n^{-2\b},\e\paren{\frac{np}{\s^2}}]$ 
\begin{equation*}
 R(n,p,\s^2,\l,\b,\d,C) \geq \a \paren{\frac{CI_1}{2\b}\l^{(2\d-1)/2\b} + \frac{\s^2I_2}{2\b np}\l^{-1/2\b}} \pt
\end{equation*}
\end{proof}

\subsection{Proof of Property~\ref{prop.min.risk} \label{sec.proof.prop.min.risk}}

\begin{proof}
 This proof uses two results proved in Sections~\ref{sec.proof.prop.min.risk1} and \ref{sec.proof.prop.min.risk2} of the appendix. Property~\ref{prop.param.reg.maj} shows that $R$ attains its minimum on $\croch{0,\e\paren{\frac{np}{\s^2}}}$, where $\e(x)$ goes to $0$ as $x$ goes to $0$. First, Property~\ref{prop.min.risk.l.small} shows that 
  \begin{equation*}
   \inf_{\l\in\croch{0,n^{-2\b}}} \set{R(n,p,\s^2,\l,\b,\d,C)} \geq \frac{\s^2}{4p} \pt
  \end{equation*}
 Then, using Property~\ref{prop.min.risk.l.large} shows that there exists an integer $N$ and a constant $\a\in(0,1)$ such that for every $(n,p,\s^2)$ such that $\frac{np}{\s^2} \geq N$, every $(\b,\d)\in\R_+^2$ such that $4\b>2\d>1$ and every $\l \in [n^{-2\b},\e\paren{\frac{np}{\s^2}}]$ we have 
 \begin{equation} \label{eq.min.R}
  R(n,p,\s^2,\l,\b,\d,C) \geq \a \paren{\frac{CI_1}{\b }\l^{(2\d-1)/2\b} + \frac{\s^2I_2}{2\b np}\l^{-1/2\b}} \pt
 \end{equation}
 Thus, using the same analysis than for Property~\ref{prop.maj.risk}, we get 
  \begin{equation*}
   \inf_{\l\in\croch{n^{-2\b},\e\paren{\frac{np}{\s^2}}}} \set{R(n,p,\s^2,\l,\b,\d,C)} \geq  \a\paren{\frac{np}{\s^2}}^{1/2\d-1}C^{1/2\d}\kbd\pt
  \end{equation*}
\end{proof}

\subsection{Proof of Property~\ref{prop.param.reg.min} \label{sec.proof.prop.param.reg.min}}

The proof of Property~\ref{prop.maj.risk} clearly shows two regimes :
\begin{itemize}
 \item when $\l^{\star}_R \leq n^{-2\b}$, the multi-task risk is $\asymp\frac{\s^2}{p}$;
 \item when $\l^{\star}_R \geq n^{-2\b}$,  the multi-task risk is $\asymp\paren{\frac{\s^2}{np}}^{1-1/2\d}$.
\end{itemize}

We now show that if $\l$ is too close to zero then the variance term exceeds the upper bound on  $R^{\star}(n,p,\s^2,\b,\d,C)$ given in Equation~\eqref{eq.maj.R.opt}.

\begin{proof}
 Let us denote 
 \begin{equation*}
   m_1 = \inf_{\l\in\croch{0,n^{-2\b}}} \set{R(n,p,\s^2,\l,\b,\d,C)}
 \end{equation*}
 and 
 \begin{equation*}
  m_2 =  \inf_{\l\in\croch{n^{-2\b},\e\paren{\frac{np}{\s^2}}}} \set{R(n,p,\s^2,\l,\b,\d,C)} \pt
 \end{equation*}
If $m_1 < m_2$  then $\l^{\star}_R \leq \frac{1}{n^{2\b}}$, else $\l^{\star}_R \geq \frac{1}{n^{2\b}}$. Under the present assumptions, we can use the proof Property~\ref{prop.min.risk} and state that there exists an integer $N_1$ and a constant $\a\in(0,1)$ such that 
 \begin{equation*}
  \frac{\s^2}{p} \geq m_1 \geq \frac{\s^2}{4p} \virg
 \end{equation*}
 and 
  \begin{equation*}
   2^{1/2\d}\paren{\frac{np}{\s^2}}^{1/2\d-1}C^{1/2\d}\kbd \geq m_2 \geq  \a\paren{\frac{np}{\s^2}}^{1/2\d-1}C^{1/2\d}\kbd\pt
  \end{equation*}
 Both assumptions  $n^{2\d-1}  \times \frac{\s^2}{p} \longrightarrow 0$ and  $n^{2\d-1} \times \frac{\s^2}{p} \longrightarrow +\infty$ ensure that either $m_1>m_2$ or $m_2<m_1$ asymptotically hold.  
\end{proof}

\section{Study of the different multi-task hypotheses}

\begin{lemma}\label{lemma.risk.2Points}
 Under Assupmtion~\eqref{hyp.AV}, Assumption \eqref{2Points} is equivalent to 
 \begin{equation*} 	
  \begin{split}
    \exists (\e_i)_{i\in\NN}\in\set{-1,1}^{\NN},~ \forall i \in \setn,~~~~~~~~~~~~~~~~~~~~~~~~~~~~~~~~~~~~~~~~~~~~~~ \\ \left\{ \begin{array}{rcl} \forall j \in \set{1,\dots,\frac{p}{2}},~ h_i^j&=& \sqrt{n}i^{-\d}(\sqrt{C_1}+\e_i\sqrt{C_2}) \\ \forall j \in \set{\frac{p}{2}+1,\dots,p},~ h_i^j &=& \sqrt{n}i^{-\d}(\sqrt{C_1}-\e_i\sqrt{C_2}) \end{array} \right. \pt
  \end{split}
  \end{equation*}
 The risk of the estimator $\fh^j_{\l}=A_{\l}y^j$ for the $j$th task, which we denote by $R^j(\l)$, verifies 
 \begin{equation*}
  R(n,1,\s^2,\l,\b,\d,\paren{\sqrt{C_1}-\sqrt{C_2}}^2) \leq R^j(\l) 
 \end{equation*}
 and 
 \begin{equation*}
    R^j(\l) \leq  R(n,1,\s^2,\l,\b,\d,\paren{\sqrt{C_1}+\sqrt{C_2}}^2) \pt
 \end{equation*}
\end{lemma}

\begin{proof}
 We have that, for every $i\in\setn$
 \begin{align*}
  &\left\{ \begin{array}{rcl} \frac{\mu_i}{\sqrt{p}} &=& \frac{1}{2} h_i^1 + \frac{1}{2}h_i^p \\ \varsigma_i^2 &=& \frac{1}{2} \paren{h_i^1- \frac{\mu_i}{\sqrt{p}}}^2 + \frac{1}{2} \paren{h_i^p-\frac{\mu_i}{\sqrt{p}}}^2 \end{array}\right. \\
 \Leftrightarrow &\left\{\begin{array}{rcl} h_i^p &=& 2\frac{\mu_i}{\sqrt{p}} - h_i^1 \\ \varsigma_i^2 &=& \frac{1}{2} \paren{h_i^1-\frac{\mu_i}{\sqrt{p}}}^2 + \frac{1}{2} \paren{2\frac{\mu_i}{\sqrt{p}} - h_i^1-\frac{\mu_i}{\sqrt{p}}}^2 \end{array}\right. \\
 \Leftrightarrow &\left\{\begin{array}{rcl} h_i^p &=& 2\mu_i - h_i^1 \\ \varsigma_i^2 &=&  \paren{h_i^1-\mu_i}^2  \end{array}\right. \\
\end{align*}
This is equivalent to  $h_i^1 = \frac{\mu_i}{\sqrt{p}} +  \varsigma_i$ and $h_i^p = \frac{\mu_i}{\sqrt{p}} -  \varsigma_i$. Thus, the first point is proved. For the second point, let $j\in\setp$. There exists $(\e_i)_{i\in\NN}\in\set{-1,1}^{\NN}$ such that $(h_i^j)^2 = ni^{-2\d} \paren{\sqrt{C_1}+\e_i\sqrt{C_2}}^2$.
The risk of $\fh^j_{\l}$ then is 
\begin{equation*}
 \l^2 \sum_{i=1}^n\frac{i^{4\b-2\d}\paren{\sqrt{C_1}+\e_i\sqrt{C_2}}^2}{(1+\l i^{2\b})^2}+ \frac{\s^2}{n} \sum_{i=1}^n \frac{1}{\left(1+\l i^{2\b}\right)^2} \pt
\end{equation*}
We conclude by seeing that, for every $\e\in\set{-1,1}$, we have $\paren{\sqrt{C_1}-\sqrt{C_2}}^2 \leq \paren{\sqrt{C_1}+\e\sqrt{C_2}}^2\leq \paren{\sqrt{C_1}+\sqrt{C_2}}^2$
\end{proof}

\begin{lemma}\label{lemma.risk.1Out}
 Under Assupmtion~\eqref{hyp.AV}, Assumption \eqref{1Out} is equivalent to 
 \begin{equation*} 
   \begin{split}
      \exists (\e_i){i\in\NN}\in\set{-1,1}^{\NN},~ \forall i \in \setn, ~~~~~~~~~~~~~~~~~~~~~~~~~~~~~~~~~~~~~~~~~~~~~~\\ \left\{ \begin{array}{rcl} \forall j \in \set{1,\dots,p-1},~ h_i^j&=& \sqrt{n}i^{-\d}\paren{\sqrt{C_1} + \e_i\sqrt{\frac{C_2}{p-1}}} \\  h_i^p &=& \sqrt{n}i^{-\d}(\sqrt{C_1} -\e_i\sqrt{(p-1)C_2}) \end{array} \right. \pt
    \end{split}
  \end{equation*}
 If $j\in\set{1,\dots,p-1}$, the risk of the estimator $\fh^j_{\l}=A_{\l}y^j$ for the $j$th task, which we denote by $R^j(\l)$, verifies 
 \begin{equation*}
  R(n,1,\s^2,\l,\b,\d,\paren{\sqrt{C_1}-\sqrt{\frac{C_2}{p-1}}}^2) \leq R^j(\l) 
 \end{equation*}
  and
 \begin{equation*}
   R^j(\l) \leq  R(n,1,\s^2,\l,\b,\d,\paren{\sqrt{C_1}+\sqrt{\frac{C_2}{p-1}}}^2) \virg
 \end{equation*}
 while the risk of the estimator $\fh^p_{\l}=A_{\l}y^p$ for the $p$th task, which is denoted by $R^p(\l)$, verifies 
 \begin{equation*}
  R(n,1,\s^2,\l,\b,\d,\paren{\sqrt{C_1}-\sqrt{(p-1)C_2}}^2) \leq R^p(\l)
 \end{equation*}
  and
 \begin{equation*}
  R^p(\l) \leq  R(n,1,\s^2,\l,\b,\d,\paren{\sqrt{C_1}+\sqrt{(p-1)C_2}}^2) \pt
 \end{equation*}
\end{lemma}

\begin{proof}
  For the first part, we  have that, for every $i\in\setn$
 \begin{align*}
  &\left\{ \begin{array}{rcl} \frac{\mu_i}{\sqrt{p}} &=& \frac{p-1}{p} h_i^1 + \frac{1}{p}h_i^p \\ \varsigma_i^2 &=& \frac{p-1}{p} \paren{h_i^1-\frac{\mu_i}{\sqrt{p}}}^2 + \frac{1}{p} \paren{h_i^p-\frac{\mu_i}{\sqrt{p}}}^2 \end{array}\right. \\
 \Leftrightarrow &\left\{\begin{array}{rcl} h_i^p &=& p\frac{\mu_i}{\sqrt{p}} - (p-1)h_i^1 \\ \varsigma_i^2 &=& \frac{p-1}{p} \paren{h_i^1-\frac{\mu_i}{\sqrt{p}}}^2 + \frac{1}{p} \paren{p\frac{\mu_i}{\sqrt{p}} - (p-1)h_i^1-\frac{\mu_i}{\sqrt{p}}}^2 \end{array}\right. \\
 \Leftrightarrow &\left\{\begin{array}{rcl} h_i^p &=& p\frac{\mu_i}{\sqrt{p}} - (p-1)h_i^1 \\ \varsigma_i^2 &=& \frac{p-1}{p} \paren{h_i^1-\frac{\mu_i}{\sqrt{p}}}^2 + \frac{(p-1)^2}{p} \paren{h_i^1-\frac{\mu_i}{\sqrt{p}}}^2 \end{array}\right. \\
 \Leftrightarrow &\left\{\begin{array}{rcl} h_i^p &=& p\frac{\mu_i}{\sqrt{p}} - (p-1)h_i^1 \\ \varsigma_i^2 &=& (p-1) \paren{h_i^1-\frac{\mu_i}{\sqrt{p}}}^2  \end{array}\right. \\
\end{align*}
This is equivalent to saying that there exists   $(\e_i)_{i\in\NN}\in\set{-1,1}^{\NN}$ such that 
\begin{align*}
 &\left\{\begin{array}{rcl} h_i^p &=& p\frac{\mu_i}{\sqrt{p}} - (p-1)h_i^1 \\ h_i^1 &=& \frac{\mu_i}{\sqrt{p}} + \frac{\e_i}{\sqrt{p-1}}\varsigma_i  \end{array}\right. \\
 \Leftrightarrow &\left\{\begin{array}{rcl}  h_i^1 &=& \frac{\mu_i}{\sqrt{p}} + \frac{\e_i}{\sqrt{p-1}}\varsigma_i  \\ h_i^p &=& \frac{\mu_i}{\sqrt{p}} -\e_i\sqrt{p-1}\varsigma_i \end{array}\right.
\end{align*}
\end{proof}

\begin{lemma}\label{lemma.hyp.2Points}
 Assumption \eqref{hyp.2Points} implies Assumption \eqref{hyp.AV}.
\end{lemma}

\begin{proof}
 For every $i\in\setn$, we suppose we have 
 \begin{equation*}
  \left\{ \begin{array}{rcl} h_i^1 &=& \sqrt{n}i^{-\d}(\sqrt{C_1}+\sqrt{C_2}) \\ h_i^p &=& \sqrt{n}i^{-\d}(\sqrt{C_1}-\sqrt{C_2}) \end{array} \right. \pt
 \end{equation*}
 Thus, 
 \begin{equation*}
  \mu_i = \frac{1}{\sqrt{p}}\sum_{j=1}^p h_i^j = \frac{\sqrt{p}}{2}\paren{h_i^1+h_i^p} = \sqrt{p}\times \sqrt{n}i^{-\d}\sqrt{C_1} \virg
 \end{equation*}
 so that $\mu_i^2 = pC_1 n i^{-2\d}$. Furthermore, 
 \begin{equation*}
  \varsigma_i^2 = \frac{1}{p}\sum_{j=1}^p\paren{h_i^j - \frac{\mu_i}{\sqrt{p}}}^2 = \frac{1}{p}\sum_{j=1}^p\paren{\sqrt{n}i^{-\d}\sqrt{C_2}}^2 = C_2ni^{-2\d} \pt
 \end{equation*}

\end{proof}

\begin{lemma}\label{lemma.hyp.21Out}
 Assumption \eqref{hyp.1Out} implies Assumption \eqref{hyp.AV}.
\end{lemma}

\begin{proof}
  For every $i\in\setn$, we suppose we have 
 \begin{equation*}
   \left \{ \begin{array}{rcl} h_i^1 &=& \sqrt{n}i^{-\d}\paren{\sqrt{C_1} + \frac{1}{\sqrt{p-1}}\sqrt{C_2}} \\
   h_i^p &=&  \sqrt{n}i^{-\d}\paren{\sqrt{C_1} -\sqrt{p-1}\sqrt{C_2}} \end{array} \right.\pt
 \end{equation*}
 Thus, 
 \begin{equation*}
  \mu_i = \frac{1}{\sqrt{p}} \sum_{j=1}^p h_i^j =  \frac{1}{\sqrt{p}}\paren{(p-1)h_i^1+h_i^p} = \sqrt{p}\times \sqrt{n}i^{-\d}\sqrt{C_1} \virg
 \end{equation*}
 so that $\mu_i^2 = pC_1 n i^{-2\d}$. Furthermore, 
  \begin{align*}
  \varsigma_i^2 &= \frac{1}{p}\sum_{j=1}^p\paren{h_i^j - \frac{\mu_i}{\sqrt{p}}}^2 \\
  &= \frac{1}{p}\croch{(p-1)\paren{\sqrt{n}i^{-\d}\frac{\sqrt{C_2}}{\sqrt{p-1}}}^2+\paren{\sqrt{p-1}\sqrt{n}i^{-\d}\sqrt{C_2}}^2} = C_2ni^{-2\d} \pt
 \end{align*}
\end{proof}

\end{document}